\documentclass[reqno, 11pt]{amsart}

\usepackage[text={155mm,235mm},centering]{geometry}
\input diagxy
\input xy

\usepackage[active]{srcltx} 

\newtheorem{thm}{Theorem}[section]
\newtheorem{cor}[thm]{Corollary}
\newtheorem{lem}[thm]{Lemma}
\newtheorem{prop}[thm]{Proposition}
\newtheorem{quest}[thm]{Question}

\newtheorem{prob}[thm]{Problem}
\theoremstyle{definition}
\newtheorem{defn}[thm]{Definition}
\theoremstyle{property}

\theoremstyle{remark}
\newtheorem{rem}[thm]{Remark}

\newtheorem{ex}[thm]{Example}
\numberwithin{equation}{section}
\usepackage[mathscr]{eucal}
\usepackage[all]{xy}
\usepackage{mathrsfs}
\usepackage{xypic}
\usepackage{amsfonts}
\usepackage{amsmath}
\usepackage{amsthm,bm}
\usepackage{amssymb,tikz-cd}
\usepackage{latexsym}
\usepackage{tabularx}
\usepackage{graphicx}
\usepackage{pict2e}
\usepackage[pagebackref]{hyperref}
\usepackage{tikz}
\usepackage{textcomp}
\usepackage[scr=rsfs]{mathalfa}
\definecolor{ceruleanblue}{rgb}{0.16, 0.32, 0.75}
\hypersetup{colorlinks=true,allcolors=ceruleanblue}
\allowdisplaybreaks

\DeclareMathOperator{\A}{\mathscr{A}}

\DeclareMathOperator{\K}{\mathscr{K}}
\def\Div{{\rm Div}}
\def\Bl{{\rm Bl}}
\def\rk{{\rm rk}}
\begin{document}

\title[Dolbeault cohomologies of blowing up complex manifolds II]{Dolbeault cohomologies of blowing up complex manifolds II: bundle-valued case}

\author{Sheng Rao}
\address{School of Mathematics and Statistics, Wuhan University, Wuhan 430072,  P. R. China}
\email{likeanyone@whu.edu.cn, raoshengmath@gmail.com}%

\author{Song Yang}
\address{Center for Applied Mathematics, Tianjin University, Tianjin 300072, P. R. China}%
\email{syangmath@tju.edu.cn}%

\author{Xiangdong Yang}
\address{Department of Mathematics, Chongqing University, Chongqing 401331, P. R. China}
\email{xiangdongyang2009@gmail.com; math.yang@cqu.edu.cn}


\date{\today}


\begin{abstract}
  We use a sheaf-theoretic approach to obtain a blow-up formula for Dolbeault cohomology
  groups with values in the holomorphic vector bundle over a compact complex manifold.
  As applications, we present several positive (or negative) examples associated to the vanishing
  theorems of Girbau, Kawamata-Viehweg and Green-Lazarsfeld in a uniform manner and study the blow-up invariance of some classical
  holomorphic invariants.
\end{abstract}

\maketitle

\tableofcontents
\section{Introduction}
\subsection{Background and related works}
In complex and algebraic geometries, one of the most important geometric transformations is blow-up (or blow-down) which provides us with a useful method to construct new complex manifolds or algebraic varieties from the original ones.
In particular, the blow-up operation plays a significant role in bimeromorphic geometry.
For instance, the Kodaira embedding theorem \cite{Kod54} is a remarkable application of the pointed blow-up technique;
the first example of non-algebraic Moishezon threefold and non-projective smooth proper algebraic threefold due to Hironaka \cite{Hir60} are constructed by blowing-up technique; Demailly-Paun's characterization of a manifold in the Fujiki Class ($\mathscr{C}$) in \cite{DP} is established by a finite sequence of blow-ups of the complex manifold with a K\"ahler current along smooth centers to construct a K\"ahler metric on the final blow-up.
According to the celebrated weak factorization theorem \cite{Wlo03,AKMW02}, each bimeromorphic map between compact complex manifolds can be factored into a finite sequence of blow-ups and blow-downs.
This implies that a holomorphic invariant of compact complex manifolds is a bimeromorphic invariant,
if and only if, it is stable under the blow-up transformation.
So from the bimeromorphic geometry point of view,
it is natural to study the variant of holomorphic invariants of compact complex manifolds under the blow-up operations,
or in a more general setting, to compare holomorphic invariants for holomorphic mappings such as the Dolbeault, Bott-Chern cohomology groups and so on.

The first result in this direction
can be traced back at least to Aeppli \cite{Aep56} for proper modifications.
Subsequently, from the viewpoint of algebraic geometry Deligne \cite{Del68} considered the comparison problem of cohomology for proper birational morphisms of smooth schemes over a field $k$ and proved an injectivity result.
In particular, for a proper modification $\pi:\tilde{X}\rightarrow X$
of compact K\"{a}hler manifolds, Grauert-Riemenschneider \cite{GR70} showed that the induced morphism
$$
\pi^{\ast}:H^{p,q}_{\bar{\partial}}(X)
\rightarrow
H^{p,q}_{\bar{\partial}}(\tilde{X}),
\,\,\,\mathrm{for\,\,any}\,\,p,q\geq0
$$
is injective.
In general, let $V$ be a holomorphic vector bundle over a complex manifold $X$.
Then we can define the Dolbeault cohomology groups of $X$ with values in $V$.
Furthermore, the holomorphic map $\pi:\tilde{X}\rightarrow X$ induces a morphism on cohomology groups
\begin{equation}\label{inj-1}
\pi^{\ast}:
H^{p,q}(X,V)
\rightarrow H^{p,q}(\tilde{X},\pi^{\ast}V),
\,\,\,\mathrm{for\,\,\,any}\,\,p,q\geq0.
\end{equation}
In his paper \cite{Wel74},
Wells proved that if $\pi$ is a proper surjective holomorphic mapping of complex manifolds with the same complex dimension, then the pullback of $\pi$ induces an injection on de Rham, Dolbeault cohomology groups. Moreover, the map \eqref{inj-1} for the bundle-valued Dolbeault cohomology groups is also injective.

\subsection{Summary of the results}
This work is motivated by the following problem.
\begin{prob}
Can we describe explicitly the cokernel of the injective morphism \eqref{inj-1}?
\end{prob}

Generally, it seems difficult to give an answer to the above problem for a holomorphic mapping and therefore we consider the special case: the blow-up morphism.
The purpose of this paper is twofold.
First, as a sequel to our work \cite{RYY},
we extend the Dolbeault blow-up formula \cite[Theorem 1.1]{RYY} to the bundle-valued case,
which appears a non-trivial result even for the projective manifolds.
Secondly, the notion of relative Dolbeault sheaves
introduced in \cite{RYY}$_{v3}$
which is a uniform generalization of the ideal and canonical sheaves, is crucial in the proof of blow-up formula on compact complex manifolds
and we believe that these sheaves admit their own independent meanings.
Here we use this kind of sheaves to prove:
\begin{thm}\label{thm1}
Assume that $X$ is a compact complex manifold with complex dimension $n$.
Let $\imath:Z\hookrightarrow X$ be a closed complex submanifold with complex codimension $r\geq 2$ and let $\pi:\tilde{X}\rightarrow X$ be the blow-up of $X$ with the center $Z$.
If $W$ is a holomorphic vector bundle over $X$,
then for any $0\leq p,q\leq n$ there exists a canonical isomorphism
\begin{equation}\label{dol-bl-for}
H^{p,q}(\tilde{X},\pi^{\ast}W)
\stackrel{\phi}\longrightarrow
H^{p,q}(X,W)\oplus
\biggl(\bigoplus^{r-1}_{i=1}
H^{p-i,q-i}(Z,\imath^{\ast}W)\biggr),
\end{equation}
where $\phi$ is a linear map defined in \eqref{phi}.
\end{thm}

A subsidiary result of the above theorem is a natural isomorphism
between cohomologies of relative Dolbeault sheaves given in Lemma \ref{kersheaf-iso}: for $0\leq p,q\leq n$,
$$
H^{q}(X, \K_{X,Z}^{p}(W)) \cong H^{q}(\tilde{X}, \K_{\tilde{X},E}^{p}(\pi^{\ast}W)),
$$
whose notations we refer to Definition \ref{rds}.
This isomorphism is a natural generalization of the key isomorphism
$$
H^{q}(X, \K_{X,Z}^{0}(W)) \cong H^{q}(\tilde{X}, \K_{\tilde{X},E}^{0}(\pi^{\ast}W))
$$
with $W$ being a holomorphic line bundle over $X$
in the proof of Bertram-Ein-Lazarsfeld's vanishing theorem \cite[Theorem 1.2]{BEL91} (see also
\cite[Lemma 4.3.16]{Laz04} for the case of $a=1$ therein).

As applications, we uniformly present several positive (or negative) examples associated to the vanishing
theorems of Girbau \cite{Gir76}, Kawamata-Viehweg  \cite{Kaw82, Vie82} and Green-Lazarsfeld \cite{GL87} and study the blow-up invariance of some classical
holomorphic invariants. In particular, one obtains a new positive example to affirm that Girbau's result is possibly true for $p\neq q$ when the Chern curvature form of the associated line bundle is semi-positive and positive only on a dense open set: Let $\pi: X\rightarrow \mathbb{C}\mathrm{P}^n$ $(n\geq 3)$ be the blowing up at a point in $\mathbb{C}\mathrm{P}^n$. Then for $p\neq q$,
$H^q(X,\Omega_X^p\otimes \pi^*\mathcal{O}_{\mathbb{C}\mathrm{P}^n}(1))=0.$

More generally,
let $f: X\rightarrow Y$ be a projective and birational morphism between $n$-dimensional smooth varieties over a field of characteristic zero.
It is known that $f_{\ast} \mathcal{O}_{X}=\mathcal{O}_{Y}$.
Moreover, as a corollary of local Chow's lemma and Hironaka's resolution of singularities \cite[Main Theorem II']{Hir64}, the higher direct images $R^{i}f_{\ast}\mathcal{O}_{X}$, $i\geq 1$, of the structure sheaf $\mathcal{O}_{X}$ is zero.
Hence,
for a locally free sheaf $\mathcal{W}$ of constant rank on $Y$,
by the projection formula, we have
$$
R^if_{\ast}f^{\ast}\mathcal{W}=
\begin{cases}
 \mathcal{W},\ &\text{$i=0$,} \\
 0,\ &\text{$i\geq 1$}.
\end{cases}
$$
As a result, an application of the Leray spectral sequence implies that for all $l\geq 0$,
\begin{equation}\label{p0ql}
H^{l}(X, f^{\ast}\mathcal W)\cong H^{l}(Y, \mathcal W),
\end{equation}
which yields the case $p=0,q=l$ of Theorem \ref{thm1} over a field of characteristic zero. The case $p=n,q=l$ of Theorem  \ref{thm1}  over a field of characteristic zero follows from $f_*\Omega^n_{X}=\Omega^n_{Y}$ and $R^if_*\Omega^n_{X}=0$, $i\geq 1$ by the Grothendieck-Serre duality (cf. \cite[\S 10]{Pop15} for example).
But the case $p\neq0, n$ for Theorem \ref{thm1} can't be obtained directly by Leray spectral sequences since the relative vanishing doesn't hold necessarily as shown in \cite[Remark 2 $(\beta)$]{Tak85} then.

By Serre's GAGA principle \cite{Ser56} and Theorem \ref{thm1},
the isomorphism \eqref{dol-bl-for} holds for smooth proper varieties over $\mathbb{C}$.
Hence, abstractly by Lefschetz's principle,
the isomorphism \eqref{dol-bl-for} should hold for smooth proper varieties over a field of characteristic zero.
By a recent result of Chatzistamatiou-R\"{u}lling \cite{CR15},
$R^{i}f_{\ast}\mathcal{O}_{X}=0$, $i\geq 1$
and $f_{\ast}\mathcal{O}_{X}=\mathcal{O}_{Y}$,
hence the  isomorphism \eqref{p0ql} holds for smooth varieties over a field of positive characteristic.
This also partially verifies that the isomorphism \eqref{dol-bl-1} holds for smooth varieties over arbitrary field.

It is also important to note that the topological approach in the proof of de Rham blow-up formula can not directly apply to the Dolbeault blow-up formula.
However, the sheaf-theoretic approach still works.
Comparing with the de Rham blow-up formula \cite[Theorem 7.31]{Voi02}, one finds that a crucial step in the proof is the Thom isomorphism.
Because each closed complex submanifold $Z$ of $X$
(as a smooth submanifold) always admits a \emph{smooth} tubular neighborhood in $X$,
the Thom isomorphism holds necessarily.
However, the $\bar{\partial}$-Thom isomorphism does not hold in general, because of the absence of the \emph{holomorphic} tubular neighborhood of $Z$ in $X$.
Especially, even for compact K\"{a}hler manifolds, there exist some obstructions for the existence of the holomorphic tubular neighborhoods (cf. \cite{RN17} and references therein).

In particular, if $W$ is a trivial holomorphic line bundle over $X$, then the formula \eqref{dol-bl-for} descends to a \emph{canonical} isomorphism
\begin{equation}\label{dol-bl-1}
\phi:H^{p,q}_{\bar{\partial}}(\tilde{X})
\rightarrow
H^{p,q}_{\bar{\partial}}(X)\oplus
\biggl(\bigoplus^{r-1}_{i=1}
H^{p-i,q-i}_{\bar{\partial}}(Z)\biggr).
\end{equation}
This implies that the isomorphism of Dolbeault blow-up formula in \cite[Theorem 1.1]{RYY} is actually \emph{canonical}. Recently, Meng \cite{Men18}
extended \eqref{dol-bl-1} to non-compact complex manifolds and Stelzig \cite{Ste18} gave a different proof of this formula.
It is noteworthy that the directions of the morphisms in \eqref{dol-bl-1} are opposite to those in \cite[Theorem 7.31]{Voi02}, \cite{Men18}.
In some sense, the morphism from the sheaf-theoretic approach is the inverse of the morphism from the topological approach.

After our preprint had been distributed, Meng \cite{Men18b} informed us that he has also considered the Dolbeault blow-up formula for the bundle-valued case on non-compact manifolds by Mayer-Vietoris sequence.
\subsection{Outline of the paper}
We devote Section \ref{sec-2} to the preliminaries on the notations throughout this paper and basic sheaf theory of Dolbeault cohomology with values in a holomorphic vector bundle.
In Section \ref{sec-3}, following the steps in the proof of the Hirsch Lemma for Dolbeault cohomology by Cordero-Fernandez-Gray-Ugarte \cite{CFGU00},
we establish that with values in a holomorphic vector bundle.
In Section \ref{sec-4}, we introduce the notation of relative Dolbeault sheaves associated to the blow-up morphism.
In Section \ref{proof}, we use the preparations of Sections \ref{sec-2}-\ref{sec-4} to give the first proof of the main result (Theorem \ref{thm1}), following the full approach of \cite{RYY}$_{v3}$.
Finally, Section \ref{sec-6} contains some applications of Theorem \ref{thm1}.
Appendix \ref{app-1} is a rapid review of Borel spectral sequence of complex analytic bundles
and Appendix \ref{app-2} presents a second proof of Theorem \ref{thm1} using the same philosophy as in \cite{Ste18}.
\subsection*{Acknowledgements}
S. Rao is partially supported by the NSFC (Grant No. 11671305, 11771339).
S. Yang and X. Yang are supported by the NSFC (Grant No. 11571242, 11701414, 11701051) and the China Scholarship Council.
The authors would like to express their gratitude to the following institutes for providing them with excellent working environment and the hospitalities during their respective visits: Institute of Mathematics, Academia Sinica; Department of Mathematics at Universit\`{a} degli Studi di Milano and Department of Mathematics at Cornell University.
Last but not least, we sincerely thank  Professor V. Navarro Aznar for answering our question on Lemma \ref{key-lem-1} $(iii)$, and Dr. L. Meng, J. Stelzig for sending us their preprints \cite{Men18,Men18b} and \cite{Ste18}, respectively.

\section{Notations and sheaf theory}\label{sec-2}
In this section we introduce notation and review some basic definitions to be referred to later.

\subsection{Notations and conventions}\label{notation-con}
Throughout this paper we assume that $X$ is an $n$-dimensional compact complex manifold and $Z\subset X$ is a closed complex submanifold such that $r=\mathrm{codim}_{\mathbb{C}}Z\geq2$.
Let $\pi:\tilde{X}\rightarrow X$ be the blow-up of $X$ with the center $Z$ and $E:=\pi^{-1}(Z)$ the exceptional divisor.
We assume the following notations:
\begin{itemize}
  \item [a)]$\imath:Z\hookrightarrow X$ is the holomorphic inclusion of $Z$ in $X$;
  \item [b)] $\jmath:E\hookrightarrow \tilde{X}$ is the holomorphic inclusion of $E$ in $\tilde{X}$;
  \item [c)]$W$ is a holomorphic vector bundle over $X$ and
        $\tilde{W}=\pi^{\ast}W$ is the pullback of $W$ by $\pi$;
  \item [d)]$\imath^{\ast}W$ and
        $\jmath^{\ast}\tilde{W}$ are the restrictions of $W$ and $\tilde{W}$ on $Z$ and $E$, respectively.
\end{itemize}
In summary, we have the following commutative cube
\begin{equation*}\label{blow-up-cube}
\xymatrix@R=0.25cm@C=0.5cm{
  & \tilde{W} \ar[rr]^{} \ar'[d][dd]
      &  & W \ar[dd]        \\
  \jmath^{\ast}\tilde{W} \ar[ur]\ar[rr]\ar[dd]
      &  & \imath^{\ast}W \ar[ur]\ar[dd] \\
  & \tilde{X} \ar'[r][rr]^{\pi}
      &  & X                \\
  E \ar[rr]^{\varpi}\ar[ur]^{\jmath}
      &  & Z \ar[ur]^{\imath}        }
\end{equation*}
The bottom of the above cube is the blow-up diagram.

\subsection{Bundle-valued Dolbeault cohomology}
Let $A^{p,q}(X,W)$ be the space of smooth sections of the bundle $\wedge^{p,q}T^{*}X\otimes_{\mathbb{C}}W$.
Assume that $k=\mathrm{rank}\,(W)$ and $W|_{U_{i}}\cong U_{i}\times\mathbb{C}^{k}$ is a holomorphic local trivialization with respect to an open covering $\{U_{i}\}_{i\in I}$ of $X$.
Under the trivialization $W|_{U}\cong U\times\mathbb{C}^{k}$ an element $\alpha\in A^{p,q}(X,W)$ can be locally written as
$$
\alpha=(\alpha_{1},\cdots,\alpha_{k}),
$$
where $\alpha_{1},\cdots,\alpha_{k}$ are smooth forms of type $(p,q)$ on $U$.
We then set
$$
\bar{\partial}_{U}(\alpha)=(\bar{\partial}\alpha_{1},\cdots,\bar{\partial}\alpha_{k});
$$
it is a section of $\Omega^{p,q+1}(U)\otimes_{\mathbb{C}}W$.
Assume that $V$ is another open subset of $X$ with the holomorphic trivialization $W|_{V}\cong V\times\mathbb{C}^{k}$ such that $U\cap V\neq\emptyset$.
The holomorphic property of the trivialization means for any $\alpha\in A^{p,q}(X,W)$ the following equality holds
\begin{equation}\label{operator-E}
\bar{\partial}_{U}(\alpha)\big|_{U\cap V}=\bar{\partial}_{V}(\alpha)\big|_{U\cap V}.
\end{equation}
The equality (\ref{operator-E}) enables us to define an operator
$$
\bar{\partial}:A^{p,q}(X,W)\rightarrow A^{p,q+1}(X,W)
$$
by the condition
$
\bar{\partial}(\alpha|_{U})=\bar{\partial}_{U}(\alpha|_{U}).
$
We call $\bar{\partial}$ the {\it canonical $(0,1)$-connection} of the holomorphic bundle $W$.
It is clear that $\bar{\partial}^{2}=0$.
Therefore, for any integer $p=0,\cdots,n$, we get the so-called {\it Dolbeault complex of $(p,\bullet)$-forms with values in $W$}:
\begin{equation}\label{complex-E}
\xymatrix@C=0.5cm{
  0 \ar[r] & A^{p,0}(X,W) \ar[r]^{\bar{\partial}} & A^{p,1}(X,W)
  \ar[r]^{\quad\,\,\,\bar{\partial}} & \cdots \ar[r]^{\bar{\partial}\quad\quad\,\,}
  & A^{p,n-1}(X,W) \ar[r]^{\bar{\partial}} & A^{p,n}(X,W) \ar[r] & 0. }
\end{equation}
The $q$-th cohomology of the complex (\ref{complex-E}), denoted by $H^{p,q}(X,W)$, is called the $(p,q)$-{\it Dolbeault cohomology group with values in $W$.}

For the given complex manifold $X$ let $\mathcal{O}_{X}$ be the sheaf of holomorphic functions on $X$ and $\Omega^{p}_{X}$ the sheaf of holomorphic differentiable forms of type $(p,0)$ on $X$.
Then there exists a canonical resolution of $\mathcal{O}_{X}$-modules for the sheaf $\Omega^{p}_{X}$, namely, the \emph{Dolbeault resolution}:
\begin{equation}\label{dolbeault-reso}
\xymatrix@C=0.5cm{
  0 \ar[r] &\Omega^{p}_{X} \ar[r]^{}& \mathscr{A}^{p,0}_{X} \ar[r]^{\bar{\partial}} & \mathscr{A}^{p,1}_{X}
  \ar[r]^{\bar{\partial}} & \cdots \ar[r]^{\bar{\partial}\quad}
  & \mathscr{A}^{p,n-1}_{X} \ar[r]^{\bar{\partial}} & \mathscr{A}^{p,n}_{X} \ar[r] & 0, }
\end{equation}
where $\mathscr{A}^{p,q}_{X}$ is the sheaf of differentiable $(p,q)$-forms on $X$.
Since $W\rightarrow X$ is a holomorphic vector bundle, the sheaf of holomorphic sections of $W$, denoted by $\mathcal{O}(W)$, is a locally free sheaf of $\mathcal{O}_{X}$-modules on $X$.
By tensoring the Dolbeault resolution (\ref{dolbeault-reso}) with $\mathcal{O}(W)$, we obtain a resolution of the sheaf $\Omega^{p}_{X}\otimes_{\mathcal{O}_{X}}\mathcal{O}(W)$ (cf. \cite[Lemma 3.19 of Chapter II]{Wel08})
\begin{eqnarray}\label{E-resolution}
&& 0 \to  \Omega^{p}_{X}\otimes_{\mathcal{O}_{X}}\mathcal{O}(W) \to
  \mathscr{A}^{p,0}_{X}\otimes_{\mathcal{O}_{X}}\mathcal{O}(W)
   \xrightarrow{\bar{\partial}\otimes1}
 \mathscr{A}^{p,1}_{X}\otimes_{\mathcal{O}_{X}}\mathcal{O}(W)
\xrightarrow{\bar{\partial}\otimes1}  \nonumber\\
&& \;\;\;\;\;\;\;\;\;\;\;\;\;  \cdots \xrightarrow{\bar{\partial}\otimes1}
\mathscr{A}^{p,n-1}_{X}\otimes_{\mathcal{O}_{X}}\mathcal{O}(W)
 \xrightarrow{\bar{\partial}\otimes1}
 \mathscr{A}^{p,n}_{X}\otimes_{\mathcal{O}_{X}}\mathcal{O}(W) \to 0.
\end{eqnarray}

Note that $\Omega^{p}_{X}\otimes_{\mathcal{O}_{X}}\mathcal{O}(W)$ is isomorphic to the sheaf of the {\it holomorphic sections} of the holomorphic bundle $\wedge^{p}T^{\prime*}X\otimes_{\mathbb{C}}W$ and $\mathscr{A}^{p,q}_{X}\otimes_{\mathcal{O}_{X}}\mathcal{O}(W)$ is isomorphic to the sheaf of
differentiable sections of the differentiable bundle $\wedge^{p,q}T^{*}X\otimes_{\mathbb{C}}W$.
For the simplicity, we write $\bar{\partial}\otimes1$ as $\bar{\partial}$,
$$
\Omega^{p}_{X}(W):=\Omega^{p}_{X}\otimes_{\mathcal{O}_{X}}\mathcal{O}(W)\cong\mathcal{O}(\wedge^{p}T^{\prime*}X\otimes_{\mathbb{C}}W)
$$
and
$$
\mathscr{A}^{p,q}_{X}(W):
=
\mathscr{A}^{p,q}_{X}\otimes_{\mathcal{O}_{X}}\mathcal{O}(W)
\cong C^{\infty}(X,\wedge^{p,q}T^{*}X\otimes_{\mathbb{C}}W).
$$
Then the sequence (\ref{E-resolution}) equals to
\begin{equation*}
\xymatrix@C=0.5cm{
  0 \ar[r] &\Omega^{p}_{X}(W) \ar[r]^{j}& \mathscr{A}^{p,0}_{X}(W) \ar[r]^{\bar{\partial}} & \mathscr{A}^{p,1}_{X}(W)
  \ar[r]^{\bar{\partial}} & \cdots\cdots \ar[r]^{\bar{\partial}\,\,\,\quad}
  & \mathscr{A}^{p,n-1}_{X}(W) \ar[r]^{\bar{\partial}} & \mathscr{A}^{p,n}_{X}(W) \ar[r] & 0, }
\end{equation*}
where
$
j:\Omega^{p}_{X}(W)\hookrightarrow\mathscr{A}^{p,0}_{X}(W)
$
is the inclusion.
From definition, the $q$-th {\it sheaf cohomology} of $\Omega^{p}_{X}(W)$ is defined to be
$$
H^{q}(X,\Omega^{p}_{X}(W)):
=R^{q}\Gamma(X,\Omega^{p}_{X}(W)),
$$
where $R^{q}\Gamma$ is the $q$-th derived functor of the global section functor.
For the sheaf complex $(\mathscr{A}^{p,\bullet}_{X}(W),\bar{\partial})$, let $R^{q}\Gamma$ be the $q$-th derived functor of the global section functor of the sheaf complex and the $q$-th {\it hypercohomology} of $(\mathscr{A}^{p,\bullet}_{X}(W),\bar{\partial})$ is defined by
$$
\mathbb{H}^{q}(X,\mathscr{A}^{p,\bullet}_{X}(W)):
=R^{q}\Gamma(X,\mathscr{A}^{p,\bullet}_{X}(W)).
$$
Due to the Dolbeault-Grothendieck Lemma \cite[Proposition 2.31]{Voi02} we get that the sheaf complex $(\mathscr{A}^{p,\bullet}_{X}(W),\bar{\partial})$ is {\it exact}
and
$$
\ker\,(\bar{\partial}:\mathscr{A}^{p,0}_{X}(W)\rightarrow \mathscr{A}^{p,1}_{X}(W))
=\Omega^{p}_{X}(W).
$$
The sheaf $\Omega^{p}_{X}(W)$ can be thought of as a sheaf complex which is zero in degree different from 0, and equal to $\Omega^{p}_{X}(W)$ in degree 0;
and therefore, the inclusion $j:\Omega^{p}_{X}(W)\hookrightarrow
\mathscr{A}^{p,\bullet}_{X}(W)$ is a quasi-isomorphism of sheaf complexes.
By \cite[Corollary 8.9]{Voi02} the following isomorphism holds:
\begin{equation}\label{derived-iso}
R^{q}\Gamma(X,\Omega^{p}_{X}(W))\cong R^{q}\Gamma(X,\mathscr{A}^{p,\bullet}_{X}(W)).
\end{equation}
Moreover, from the Dolbeault theorem \cite[Theorem 3.20 of Chapter II]{Wel08} and (\ref{derived-iso}) we get
\begin{equation*}
H^{p,q}(X,W)\cong H^{q}(X,\Omega^{p}_{X}(W))\cong \mathbb{H}^{q}(X,\mathscr{A}^{p,\bullet}_{X}(W)).
\end{equation*}

\subsection{Direct and inverse images}
Here we refer to \cite{Har77} as a standard reference of sheaf theory.

Let $f:X\rightarrow Y$ be a continuous map of topological spaces.
For an abelian sheaf $\mathscr{F}$ on $X$ we define the {\it direct image} $f_{\ast}\mathscr{F}$ to be the sheaf on $Y$ defined by setting
$$
U
\mapsto
f_{\ast}\mathscr{F}(U)=\mathscr{F}(f^{-1}(U)),
$$
where $U$ is open in $Y$.
The $i$-th \emph{higher direct image} $R^{i}f_{*}\mathscr{F}$ is defined to be the sheafification of the presheaf
$$
U
\mapsto
H^{i}(f^{-1}(U),\mathscr{F})
$$
for any open $U\subset Y$; equivalently, $R^{i}f_{*}\mathscr{F}=\mathscr{H}^{i}(f_{\ast}I^{\bullet})$ for an injective resolution $I^{\bullet}$ of $\mathscr{F}$ and $\mathscr{H}^{i}(f_{\ast}I^{\bullet})$ is the $i$-th cohomology sheaf of the sheaf complex $f_{\ast}I^{\bullet}$.
In general, since the functor $f_{\ast}$ is left exact, one can define the {\it right derived functor} $Rf_{\ast}$ of $f_{\ast}$ as follows:
$$
Rf_{\ast}(\mathscr{F}^{\bullet}):=f_{\ast}I^{\bullet},
$$
where $\mathscr{F}^{\bullet}$ is a bounded blow sheaf complex and
$\mathscr{F}^{\bullet}\rightarrow I^{\bullet}$ is a quasi-isomorphism with $I^{\bullet}$ a complex of injectives.
Particularly, if $X$ and $Y$ have additional structure (for instance, ringed spaces) we mean $R^{i}f_{\ast}$ and $Rf_{\ast}$ in this sense.

Let $\mathscr{G}$ be an abelian sheaf on $Y$.
The {\it topological inverse image} of $\mathscr{G}$, denoted by $f^{-1}\mathscr{G}$, is defined to be the sheafification
of the presheaf
$$
U\mapsto
\lim_{f(U)\subseteq V}\mathscr{G}(V)
$$
for every open subset $U$ of $X$.

Assume that $f$ is a morphism of ringed spaces
$(X,\mathcal{O}_{X})\rightarrow (Y,\mathcal{O}_{Y})$.
If $\mathscr{F}$ is an $\mathcal{O}_{X}$-module,
then the derived image $R^{i}f_{\ast}\mathscr{F}$ has the structure of $\mathcal{O}_{Y}$-module.
Since $\mathcal{O}_{X}$ is an $f^{-1}\mathcal{O}_{Y}$-module,
the {\it inverse image} of $\mathscr{G}$ in the sense of ringed spaces is defined to be
$$
f^{\ast}\mathscr{G}
:=
f^{-1}\mathscr{G}\otimes_{f^{-1}
\mathcal{O}_{Y}}\mathcal{O}_{X}.
$$

Furthermore, we have the well known {\it projection formula} to be frequently used in this context.

\begin{lem}[Projection formula]
For any $i\in \mathbb{N}$,
there is a natural isomorphism
\begin{equation}\label{proj-formula}
R^{i}f_{\ast}\mathcal{E}\otimes \mathcal{F}
\stackrel{\simeq}\longrightarrow
R^{i}f_{\ast}(\mathcal{E}\otimes f^{\ast}\mathcal{F})
\end{equation}
for a sheaf of $\mathcal{O}_{X}$-module $\mathcal{E}$ and a locally free sheaf of constant rank $\mathcal{F}$ on $Y$.
\end{lem}


\section{Hirsch Lemma: bundle-valued case }
\label{sec-3}

The purpose of this section is to establish the Hirsch Lemma for bundle-valued Dolbeault cohomology
using the same results and intermediate steps as in \cite[Section 4.2]{CFGU00}, which is possibly known to experts.

Suppose that
$\xymatrix@C=0.5cm{
F\hookrightarrow E\ar[r]^{\,\,\,\quad\pi} &B }
$
is a holomorphic fibration such that $E,B,F$ are connected complex manifolds.
Moreover, we assume that the structure group $G$ of $\xi=(E,B,F,\pi)$ is connected and
all forms we consider are smooth and $\mathbb{C}$-valued.
Let $W$ be a holomorphic vector bundle over $B$ and $\tilde{W}$ the pull-back of $W$ by the projection $\pi$
$$
\xymatrix{
& \tilde{W} \ar[d]_{} \ar[r]^{} & W \ar[d]^{} \\
& E \ar[r]^{\pi} & B.  }
$$
From definition, the projection $\pi$ induces a morphism
$
\pi^{\ast}:A^{\bullet,\bullet}(X,W)\rightarrow A^{\bullet,\bullet}(\tilde{X},\pi^{*}W).
$

\begin{defn}\label{extending}
Let $j:F\hookrightarrow E$ be the inclusion mapping.
We say that $H^{\bullet,\bullet}_{\bar{\partial}}(F)$ satisfies the {\it extending condition},
if the morphism of bigraded algebras
$j^{*}:H^{\bullet,\bullet}_{\bar{\partial}}(E)
\rightarrow
H^{\bullet,\bullet}_{\bar{\partial}}(F)$
is surjective.
\end{defn}

It is noteworthy that the extending condition here is stronger than the condition that each generator of $H^{\bullet,\bullet}_{\bar{\partial}}(F)$ is \emph{transgressive} in the sense of \cite[Section 4.2]{CFGU00}.
However, the extending condition is always satisfied for the projectivization of a holomorphic vector bundle.

Assume that $H^{\bullet,\bullet}_{\bar{\partial}}(F)$ is a \emph{free} bigraded algebra satisfying the extending condition.
Pick up a basis $\{\textbf{x}_{1},\cdots,\textbf{x}_{m}\}$ of $H^{\bullet,\bullet}_{\bar{\partial}}(F)$.
From Definition \ref{extending}, there exists an element $\tilde{\textbf{x}}_{i}\in H^{\bullet,\bullet}_{\bar{\partial}}(E)$ such that $j^{*}(\tilde{\textbf{x}}_{i})=\textbf{x}_{i}$.
Let
$T=A^{\bullet,\bullet}(B,W)\otimes_{\mathbb{C}} H^{\bullet,\bullet}_{\bar{\partial}}(F)$
and define the differential $\bar{\partial}$ on $T$ by setting
\begin{equation}\label{barpartial}
\bar{\partial}(\omega\otimes \textbf{x}_{i})
=(\bar{\partial}\omega)\otimes \textbf{x}_{i},
\,\,\,\mathrm{for\,\,\,any}\,\,\,1\leq i\leq m.
\end{equation}
Thus, $(T,\bar{\partial})$ is a differential graded complex over $\mathbb{C}$.
Moreover, we can define a morphism of complexes
\begin{equation}\label{t-E-mor}
\psi:T\rightarrow A^{\bullet,\bullet}(E,\pi^{*}W),\,\,\,
\omega\otimes \textbf{x}_{i}\mapsto\pi^{*}(\omega)
\wedge\tilde{\textbf{x}}_{i},
\end{equation}
where $\tilde{\textbf{x}}_{i}$ can be viewed as a form on $E$ with values in the trivial holomorphic line bundle.

Set
$$
^{p,q}T=\sum_{a+c=p\atop{b+d=q}}
A^{a,b}(B,W)\otimes H^{c,d}_{\bar{\partial}}(F)
$$
and define
$$
L^{k}T=\sum_{a+b\geq k}A^{a,b}(B,W)\otimes H^{c,d}_{\bar{\partial}}(F),
$$
where the degree $0\leq k\leq n:=\textmd{dim}_{\mathbb{R}}B$.
Then we get
$$
T={\bigoplus_{p,q}}^{p,q}T,\,\,\,\,\,\,
\bar{\partial}(L^{k}T)\subset L^{k}T.
$$
The above construction actually yields a {\it filtration} on $T$:
\begin{equation}\label{filtration}
T=L^{0}T\supset L^{1}T\supset\cdots\supset L^{n-1}T\supset L^{n}T\supset0.
\end{equation}

We introduce the following subspaces of $T$:
\begin{itemize}
  \item [(i)] $T^{s+t}=\bigoplus_{p+q=s+t}({^{p,q}T}),$
  \item [(ii)] $^{p,q}L^{k}T=(L^{k}T)\cap(^{p,q}T),$
  \item [(iii)] $^{p,q}Z^{s,t}_{l}=(^{p,q}L^{s}T^{s+t})\cap(\bar{\partial}^{-1}\bigl(\,^{p,q+1}L^{s+l}T^{s+t+1}\bigr)),$
  \item [(iv)] $^{p,q}B^{s,t}_{l}=(^{p,q}L^{s}T^{s+t})\cap\bar{\partial}\bigl(^{p,q-1}L^{s-l}T^{s+t-1}\bigr).$
\end{itemize}
The differential $\bar{\partial}$ acts on $^{p,q}L^{k}T$ with the degree 1 in $q$ and the degree 0 in $p$, i.e.,
$$
\bar{\partial}:\,^{p,q}L^{k}T\rightarrow\,^{p,q+1}L^{k}T.
$$

According to the standard spectral sequence theory, there exists a spectral sequence $\{\hat{E}_{l},\hat{d}_{l}\}$ with respect to the filtration \eqref{filtration} such that
\begin{itemize}
  \item [(i)]the term of $\hat{E}_{l}$ of type $(p,q)$ with the filtration degree $s$ and the total degree $s+t$ is
             $$
             ^{p,q}\hat{E}^{s,t}_{l}=
             \frac{^{p,q}Z^{s,t}_{l}}{^{p,q}Z^{s+1,t-1}_{l-1}
             +{^{p,q}B^{s,t}_{l-1}}};
             $$
  \item [(ii)] the differential is
             $$
             \hat{d}_{l}:\,^{p,q}\hat{E}^{s,t}_{l}
             \rightarrow\,^{p,q+1}\hat{E}^{s+l, t-l+1}_{l};
             $$
  \item [(iii)] $\{\hat{E}_{l}\}$ converges to $H_{\bar{\partial}}(T)$, i.e.,
             $$
             \sum_{p+q=s+t}{^{p,q}\hat{E}^{s,t}_{\infty}}\cong
             \textmd{Gr}\,H^{p,q}_{\bar{\partial}}(T)\cong
             H^{p,q}_{\bar{\partial}}(T).
             $$
\end{itemize}
We claim that the second term $\hat{E}_{2}$ is
$$
^{p,q}\hat{E}^{s,t}_{2}\cong
\sum_{i\geq0}H^{i,s-i}(B,W)\otimes H^{p-i,q-s+1}_{\bar{\partial}}(F).
$$

For the $(p,q)$-type subspace of $T$ with the filtration degree $s$ and total degree $s+t$, we have
\begin{equation}\label{equ4.2}
^{p,q}L^{s}T^{s+t}=\sum_{a+b\geq s\atop{p+q=s+t}}\sum_{a+c=p\atop{b+d=q}}
A^{a,b}(B,W)\otimes H^{c,d}_{\bar{\partial}}(F)
\end{equation}
and
\begin{equation}\label{equ4.3}
^{p,q}L^{s+1}T^{s+t}=\sum_{a+b\geq s+1\atop{p+q=s+t}}\sum_{a+c=p\atop{b+d=q}}
A^{a,b}(B,W)\otimes H^{c,d}_{\bar{\partial}}(F).
\end{equation}
Set
\begin{equation}\label{equ4.4}
^{p,q}Q^{s,t}=\sum_{a+b=s\atop{p+q=s+t}}
\sum_{a+c=p\atop{b+d=q}}A^{a,b}(B,W)\otimes H^{c,d}_{\bar{\partial}}(F)
=\sum_{j\geq0}A^{j,s-j}(B,W)\otimes H^{p-j,q-s+j}_{\bar{\partial}}(F).
\end{equation}
From (\ref{equ4.2})-(\ref{equ4.4}) we obtain the following equality
\begin{equation}\label{equ4.5}
^{p,q}L^{s}T^{s+t}=\,^{p,q}Q^{s,t}\oplus\,^{p,q}L^{s+1}T^{s+t}.
\end{equation}

We are ready to compute $^{p,q}\hat{E}^{s,t}_{2}$.
From definition, we have
\begin{equation}\label{z-1}
{^{p,q}Z^{s+1,t-1}_{1}}=
(^{p,q}L^{s+1}T^{s+t})\cap\bar{\partial}^{-1}
       \bigl(^{p,q+1}L^{s+2}T^{s+t+1}\bigr).
\end{equation}
Furthermore, it is straightforward to get
\begin{eqnarray}\label{z-2}
  ^{p,q}Z^{s,t}_{2}
   &=& \,^{p,q}L^{s}T^{s+t}\cap
   \bar{\partial}^{-1}\bigl(^{p,q+1}L^{s+2}T^{s+t+1}\bigr) \nonumber\\
   &=&\bigl(^{p,q}Q^{s,t}\oplus
   \,^{p,q}L^{s+1}T^{s+t}\bigr)
   \cap\bar{\partial}^{-1}\bigl(^{p,q+1}L^{s+2}T^{s+t+1}\bigr) \quad\quad\quad(\textmd{By}
   \,\,\,\eqref{equ4.5}-\eqref{z-1})\nonumber\\
   &=& \bigl(^{p,q}Q^{s,t}\cap\bar{\partial}^{-1}
   \bigl(^{p,q+1}L^{s+2}T^{s+t+1}\bigr)\bigr)
   \oplus{^{p,q}Z^{s+1,t-1}_{1}}.
\end{eqnarray}
If
$
\omega_{s}=\sum_{j\geq0}(\omega^{j,s-j}\otimes \textbf{x}^{p-j,q-s+j})\in\,^{p,q}Q^{s,t},
$
then we have
\begin{equation}\label{equ4.6}
\bar{\partial}(\omega_{s})=
\sum_{j\geq0}\bar{\partial}(\omega^{j,s-j})
\otimes \textbf{x}^{p-j,q-s+j}.
\end{equation}
Consider the filtration degrees of the terms in the right side of (\ref{equ4.6}).
Observe that
$$
\bar{\partial}(\omega^{j,s-j})\otimes \textbf{x}^{p-j,q-s+j}\in L^{s+1}T
$$
has the filtration degree $s+1$.
This means that $\bar{\partial}(\omega_{s})\in\,^{p,q+1}L^{s+2}T^{s+t+1}$, namely,
$\bar{\partial}(\omega_{s})$ has the filtration degree at least $s+2$, if and only if $\bar{\partial}(\omega^{j,s-j})
\otimes \textbf{x}^{p-j,q-s+j}=0$ for any $j\geq0$
and hence
$\bar{\partial}(\omega^{j,s-j})=0$.
It follows that the subspace
$$^{p,q}Q^{s,t}\cap\bar{\partial}^{-1}
   \bigl(^{p,q+1}L^{s+2}T^{s+t+1}\bigr)$$
is equivalent to
$$
Y:=\biggl\{\,\sum_{j\geq0}\omega^{j,s-j}
\otimes \textbf{x}^{p-j,q-s+j}\in\,^{p,q}Q^{s,t}
\,\bigg|\,\bar{\partial}(\omega^{j,s-j})=0,
\,\,\mathrm{for\,\,\,any}\,\,\,j\geq0\biggr\}
$$
and thus, by \eqref{z-2}, we get
\begin{equation}\label{z-2-Y}
^{p,q}Z^{s,t}_{2}=Y\oplus{^{p,q}Z^{s+1,t-1}_{1}}.
\end{equation}

For the term $^{p,q}B^{s,t}_{1}$, the definition shows
$$
^{p,q}B^{s,t}_{1}= \,(^{p,q}L^{s}T^{s+t})\cap
\bar{\partial}\bigl(^{p,q-1}L^{s-1}T^{s+t-1}\bigr).
$$
An element $\omega_{s-1}\in\,^{p,q-1}L^{s-1}T^{s+t-1}$ has a unique expression
$$
\omega_{s-1}=
\sum_{i\geq0}\omega^{i,s-i-1}\otimes \textbf{x}^{p-i,q-s+i}
+
\sum_{i\geq0\atop{k\geq s}}\omega^{i,k-i}
\otimes \textbf{x}^{p-i,q-k-1+i}.
$$
Due to \eqref{barpartial}, we get
\begin{equation*}
\bar{\partial}(\omega_{s-1})=
\sum_{i\geq0}
\bar{\partial}\bigl(\omega^{i,s-i-1}\bigr)
\otimes \textbf{x}^{p-i,q-s+i}
+\sum_{i\geq0\atop{k\geq s}}\bar{\partial}(\omega^{i,k-i})
\otimes \textbf{x}^{p-i,q-k-1+i}.
\end{equation*}
On one hand, a careful degree checking shows
\begin{equation*}
\omega^{\prime}:=
\sum_{i\geq0}\bar{\partial}\bigl(\omega^{i,s-i-1}\bigr)
\otimes \textbf{x}^{p-i,q-s+i}\in\,^{p,q}L^{s}T^{s+t}
\end{equation*}
and
\begin{equation*}
\omega^{\prime\prime}:=
\sum_{i\geq0\atop{k\geq s}}\bar{\partial}\bigl(\omega^{i,k-i}\bigr)
\otimes \textbf{x}^{p-i,q-k-1+i}\in\,^{p,q}L^{s+1}T^{s+t}.
\end{equation*}
On the other hand, since
$\bar{\partial}(\omega^{\prime\prime})=0$ it means
$$
\omega^{\prime\prime}\in
(^{p,q}L^{s+1}T^{s+t})\cap
\bar{\partial}^{-1}\bigl(^{p,q+1}L^{s+2}T^{s+t+1}\bigr)
=
{^{p,q}Z^{s+1,t-1}_{1}}.
$$
As a result, we obtain
\begin{equation}\label{b-1}
^{p,q}B^{s,t}_{1}=
\bar{\partial}\,\bigl(^{p,q-1}L^{s-1}T^{s+t-1}\bigr)+
\mathrm{terms\,\,\,in}\,\,{^{p,q}Z^{s+1,t-1}_{1}}.
\end{equation}
From \eqref{z-2-Y} and \eqref{b-1}, it concludes the isomorphisms:
$$
^{p,q}\hat{E}^{s,t}_{2}
=\frac{^{p,q}Z^{s,t}_{2}}{^{p,q}Z^{s+1,t-1}_{1}
+{^{p,q}B^{s,t}_{1}}}
\cong\frac{Y}
{\bar{\partial}\,\bigl(^{p,q-1}L^{s-1}T^{s+t-1}\bigr)},
$$
and thus
\begin{equation}\label{hat-e2}
^{p,q}\hat{E}^{s,t}_{2}\cong
\sum_{i\geq0}H^{i,s-i}_{\bar{\partial}}(B,W)
\otimes
H^{p-i,q-s+i}_{\bar{\partial}}(F).
\end{equation}

\begin{lem}[Hirsch Lemma]\label{hirsch-lem}
Suppose that $H^{\bullet,\bullet}_{\bar{\partial}}(F)$
satisfies the extending condition (see Definition \ref{extending}).
Then the morphism $\psi$ in \eqref{t-E-mor} induces an isomorphism
$$
H^{\bullet,\bullet}_{\bar{\partial}}(T)\cong H^{\bullet,\bullet}(E,\pi^{*}W).
$$
\end{lem}
\begin{proof}
Let $L^{\ast}A^{\bullet,\bullet}(E,\pi^{*}W)$ be the Borel filtration of $A^{\bullet,\bullet}(E,\pi^{*}W)$.
On one hand, according to Theorem \ref{thm3.2} and Corollary \ref{cor3.3} we get that the associated Borel spectral sequence $\{E_{l},d_{l}\}$ converges to $H(E,\pi^{*}W)$ with the second terms
\begin{equation}\label{e2}
^{p,q}E^{s,t}_{2}\cong\sum_{i\geq0}H^{i,s-i}(B,W)\otimes H^{p-i,q-s+1}_{\bar{\partial}}(F).
\end{equation}
On the other hand, note that the morphism of complexes $\psi:T\rightarrow A^{\bullet,\bullet}(E,\pi^{*}W)$ is filtration preserving,
and hence it induces a morphism of spectral sequences
$
\psi_{l}:\hat{E}_{l}\rightarrow E_{l}.
$
Combining \eqref{hat-e2} with \eqref{e2} implies that
$$
\psi_{2}:{^{p,q}\hat{E}^{s,t}_{2}}\rightarrow{^{p,q}E^{s,t}_{2}}
$$
is isomorphic.
As a direct consequence of a standard result in spectral sequence theory we get $\hat{E}_{l}\cong E_{l}$ for any $l\geq2$.
This means $\hat{E}_{\infty}\cong E_{\infty}$, and therefore $\psi$ induces an isomorphism of cohomologies
$
H^{\bullet,\bullet}_{\bar{\partial}}(T)\cong H^{\bullet,\bullet}(E,\tilde{W}),
$
which completes the proof.
\end{proof}

Now consider the projectivization
$\pi:\mathbb{P}(V)\rightarrow B$ of a holomorphic vector bundle $V$ over $B$ with rank $r$.
Then the fiber $F$ of $\mathbb{P}(V)$ is $\mathbb{C}\mathrm{P}^{r-1}$.
Assume that $W\rightarrow B$ is a holomorphic vector bundle and let $\tilde{W}=\pi^{\ast}W$ be the pull-back of $W$ on $\mathbb{P}(V)$.
Then we have
\begin{lem}\label{dolb-projective-formula}
The Dolbeault cohomology of $\tilde{W}$ is a free $H^{\bullet,\bullet}(B,W)$-bigraded module with the basis $\{1,\tilde{\bm{t}},\cdots,\tilde{\bm{t}}^{r-1}\}$;
namely, we have the canonical isomorphism
$$
H^{\bullet,\bullet}(\mathbb{P}(V),\tilde{W})
\cong H^{\bullet,\bullet}(B,W)\otimes\{1,\tilde{\bm{t}},\cdots,
\tilde{\bm{t}}^{r-1}\},
$$
where $\tilde{\bm{t}}=c_{1}(\mathcal{O}_{\mathbb{P}(V)}(1))\in H^{1,1}_{\bar{\partial}}(\mathbb{P}(V))$.
\end{lem}
\begin{proof}
Note that the Dolbeault cohomology of $F$ is
$$
H^{\bullet,\bullet}_{\bar{\partial}}(F)\cong H^{\bullet,\bullet}_{\bar{\partial}}(\mathbb{C}\mathrm{P}^{r-1})
\cong\mathbb{C}[\bm{t}]/(\bm{t}^{r}),
$$
where the generator $\bm{t}\in H^{1,1}_{\bar{\partial}}(\mathbb{C}\mathrm{P}^{r-1})$ is the K\"{a}hler form of the Fubini-Study metric on $\mathbb{C}\mathrm{P}^{r-1}$.
We claim that $H^{\bullet,\bullet}_{\bar{\partial}}(F)$ satisfies the extending condition.
Consider the first Chern class of the tautological line bundle over $\mathbb{P}(V)$:
$$
\tilde{\bm{t}}=c_{1}(\mathcal{O}_{\mathbb{P}(V)}(1))\in H^{1,1}_{\bar{\partial}}(\mathbb{P}(V)).
$$
From definition, the restriction of $\tilde{\bm{t}}$ to each fiber $F$ is the generator of $H^{*,*}_{\bar{\partial}}(F)$, namely, $\tilde{\bm{t}}|_{F}=\bm{t}$.
Consequently, $\bm{t}$ can be extended to be a class on $E$.
Set
$
T=A^{\bullet,\bullet}(B,W)\otimes H^{\bullet,\bullet}_{\bar{\partial}}(\mathbb{C}\mathrm{P}^{r-1}).
$
We can define a natural differential operator of (0,1)-type on $T$ by setting
$$
\bar{\partial}_{T}(a\otimes \bm{t})=(\bar{\partial}a)\otimes \bm{t},
$$
for any $a\otimes \bm{t}\in T$.
By definition, the cohomology of $T$ is
\begin{eqnarray*}
  H^{\bullet,\bullet}(T,\bar{\partial}_{T})
  &=& H^{\bullet,\bullet}(B,W)\otimes H^{\bullet,\bullet}_{\bar{\partial}}(\mathbb{C}\mathrm{P}^{r-1}) \\
  &\cong& H^{\bullet,\bullet}(B,W)\otimes \{1,\bm{t},\cdots,\bm{t}^{r-1}\}\\
  &\cong& H^{\bullet,\bullet}(B,W)\otimes \{1,\tilde{\bm{t}},\cdots,\tilde{\bm{t}}^{r-1}\}.
\end{eqnarray*}
According to Lemma \ref{hirsch-lem} we obtain
$$
H^{\bullet,\bullet}(\mathbb{P}(V),\tilde{W})
\cong H^{\bullet,\bullet}(T,\bar{\partial}_{T})
\cong H^{\bullet,\bullet}(B,W)\otimes\{1,\tilde{\bm{t}},
\cdots,\tilde{\bm{t}}^{r-1}\},
$$
and the proof is complete.
\end{proof}
\section{Relative Dolbeault sheaves associated to the blow-up diagram}\label{sec-4}
\subsection{Higher direct images}
Recall the blow-up diagram:
\begin{equation}\label{blowup-diag3}
\xymatrix{
E \ar[d]_{\varpi} \ar@{^{(}->}[r]^{\jmath} & \tilde{X}\ar[d]^{\pi}\\
 Z \ar@{^{(}->}[r]^{\imath} & X.}
\end{equation}
Then we have the following isomorphisms:
\begin{equation}\label{3-iso}
\pi^{\ast}:
\Omega^{p}_{X}
\stackrel{\simeq}\longrightarrow
\pi_{\ast}\Omega^{p}_{\tilde{X}},\,\,\,
\varpi^{\ast}:
\Omega^{p}_{Z}
\stackrel{\simeq}\longrightarrow
\varpi_{\ast}
\Omega^{p}_{E},\,\,\,
\textmd{and}\,\,\,
\jmath^{\ast}:
R^{s}\pi_{\ast}\Omega^{p}_{\tilde{X}}
\stackrel{\simeq}\longrightarrow
\imath_{\ast}R^{s}\varpi_{\ast}\Omega^{p}_{E} \,\,\,(s\geq1).
\end{equation}
The first isomorphism in \eqref{3-iso} can be proved by using the Hartogs extension theorem
and for the second we refer to \cite[\S 4, Theorem 2]{Ver74}.
For $s>1$, the third isomorphism in \eqref{3-iso} was first proved by Gros \cite{Gro85} for smooth schemes over a field.
Subsequently, in their paper \cite{GNA02}, Gull\'{e}n-Navarro Aznar  improved this result for any $s\geq1$ on smooth schemes over a field.
In his paper \cite{Ste18}, Stelzig indicated those isomorphisms for complex manifolds.
Without claiming any originality,
we extend the isomorphisms in \eqref{3-iso} to the bundle-valued case by following the same steps in the proof of \cite[Proposition 3.3]{GNA02} since there is no available literature.

\begin{lem}\label{key-lem-1} For a holomorphic vector bundle $W$ on $X$ and $\tilde{W}:=\pi^{\ast}W$,
there hold:
\begin{itemize}
  \item [(i)]
          $
          \pi^{\ast}:
          \Omega^{p}_{X}(W)
          \stackrel{\simeq}\longrightarrow
          \pi_{\ast}\Omega^{p}_{\tilde{X}}(\tilde{W})
          $,
  \item [(ii)]
          $
           \varpi^{\ast}:
           \Omega^{p}_{Z}(\imath^{\ast}W)
           \stackrel{\simeq}\longrightarrow
           \varpi_{\ast}\Omega^{p}_{E}(\jmath^{\ast}\tilde{W})
           $,
  \item [(iii)]
      $\jmath^{\ast}:
      R^{s}\pi_{\ast}\Omega^{p}_{\tilde{X}}(\tilde{W})
      \stackrel{\simeq}\longrightarrow
      \imath_{\ast}R^{s}\varpi_{\ast} \Omega^{p}_{E}(\jmath^{\ast}\tilde{W}), \,\,\,    (s\geq1).
      $
\end{itemize}
\end{lem}

\begin{proof}
(i)
The morphism
$
\pi^{\ast}: \Omega_{X}^{p}
\rightarrow
\pi_{\ast}\Omega_{\tilde{X}}^{p}
$
is naturally induced by the pullback of holomorphic differential forms
\begin{eqnarray*}
 \pi^{\ast}(V):\Gamma(V, \Omega_{X}^{p})
  &\rightarrow& \Gamma(\tilde{V}, \Omega_{\tilde{X}}^{p}) \\
  \alpha &\mapsto& (\pi|_{\tilde{V}})^{\ast} \alpha
\end{eqnarray*}
for every open subset $V\subset X$ and $\tilde{V}:=\pi^{-1}(V)$.

First we prove that the morphism
$
\pi^{\ast}: \Omega_{X}^{p}
\rightarrow
\pi_{\ast}\Omega_{\tilde{X}}^{p}
$ is isomorphic.
It suffices to show that $\pi^{\ast}(V)$ is isomorphic for any open set $V\subset X$ with $V\cap Z\neq \emptyset$ since $\pi: \tilde{X}-E\rightarrow X-Z$ is biholomorphic.
It is easy to see that $\pi^{\ast}(V)$ is injective.
In fact,
note that
$
\pi|_{\tilde{V}-\tilde{V}\cap E}: \tilde{V}-\tilde{V}\cap E \rightarrow V-V\cap Z
$
is biholomorphic.
Suppose $\beta\in \Gamma(V, \Omega_{X}^{p})$ with $(\pi|_{\tilde{V}})^{\ast}\beta=0$.
Then we obtain $\beta|_{V-V\cap Z}=(\pi|_{\tilde{V}-\tilde{V}\cap E})^{\ast}\beta=0$.
Since $\mathrm{codim}_{\mathbb{C}}V\cap Z \geq 2$,
the continuity argument gives $\beta=0$.
Moreover,
for any
$
\tilde{\alpha}
\in
\Gamma(\tilde{V},\Omega_{\tilde{X}}^{p}),
$
we define a holomorphic $p$-form on $V-V\cap Z$:
$$
\varphi:=((\pi|_{\tilde{V}- \tilde{V}\cap E})^{-1})^{\ast}(\tilde{\alpha}|_{\tilde{V}-\tilde{V}\cap E}).
$$
As $\mathrm{codim}_{\mathbb{C}}V\cap Z \geq 2$,
the Hartogs extension theorem yields a holomorphic extension $\alpha$ of $\varphi$ on $V$
such that $\alpha|_{V-V\cap Z}=\varphi$.
Then we have
$
\big((\pi|_{\tilde{V}})^{\ast}\alpha\big)|_{\tilde{V}-\tilde{V}\cap E}
=
(\pi|_{\tilde{V}- \tilde{V}\cap E})^{\ast}\varphi
=
\tilde{\alpha}|_{\tilde{V}-\tilde{V}\cap E}
$.
Again by the continuity argument, we have $(\pi|_{\tilde{V}})^{\ast}\alpha=\tilde{\alpha}$. This shows that $\pi^{\ast}(V)$ is surjective
and thus $\pi^{\ast}(V)$ is isomorphic.

Furthermore, note that
$
\Omega^{p}_{\tilde{X}}(\tilde{W})
=
\Omega^{p}_{\tilde{X}}\otimes\tilde{W}
$
and according to the projection formula, we get
$$
\pi_{\ast}\Omega_{\tilde{X}}^{p}\otimes W
\stackrel{\simeq}\longrightarrow
\pi_{\ast}(\Omega_{\tilde{X}}^{p}\otimes\pi^{\ast}W)
=\pi_{\ast}(\Omega_{\tilde{X}}^{p}\otimes\tilde{W}).
$$

(ii)
The morphism
$
\varpi^{\ast}: \Omega_{Z}^{p}
\rightarrow
\varpi_{\ast}\Omega_{E}^{p}
$
is naturally induced by the pullback of holomorphic differential forms
$$
\varpi^{\ast}(V): \Gamma(V, \Omega_{Z}^{p})
\longrightarrow
\Gamma(\varpi^{-1}(V), \Omega_{E}^{p})
$$
for any open subset $V\subset Z$.
The local trivialization of fibre bundle gives an open neighborhood for any point in $Z$, still denoted by $V$, such that
$\varpi^{-1}(V)$
is biholomorphic to
$V\times \mathbb{C}\mathrm{P}^{r-1}.$
Notice that
$$
\Gamma(\varpi^{-1}(V),\Omega^{p}_{E})
\cong
H^{p,0}_{\bar{\partial}}(V\times \mathbb{C}\mathrm{P}^{r-1}).
$$
From the K\"{u}nneth formula \cite[Corollary 19]{CFGU00} (still true with the noncompact base),
we get
\begin{eqnarray*}
  H^{p,0}_{\bar{\partial}}(V\times \mathbb{C}\mathrm{P}^{r-1})
  &\cong& \sum_{a+c=p}H^{a,0}_{\bar{\partial}}(V)
  \otimes H^{c,0}_{\bar{\partial}}(\mathbb{C}\mathrm{P}^{r-1}) \\
  &=& \Gamma(V,\Omega^{p}_{Z}).
\end{eqnarray*}
So
$
\varpi^{\ast}: \Omega_{Z}^{p}
\stackrel{\simeq}\longrightarrow
\varpi_{\ast}\Omega_{E}^{p}.
$
Moreover, we have
\begin{eqnarray*}
\varpi_{\ast}(\Omega_{E}^{p}\otimes \jmath^{\ast}\tilde{W})
&=& \varpi_{\ast}(\Omega_{E}^{p}\otimes \jmath^{\ast}\pi^{\ast}W)\\
&=& \varpi_{\ast}(\Omega_{E}^{p}\otimes \varpi^{\ast}\imath^{\ast}W) \;\;\quad\quad(\varpi^{\ast}\imath^{\ast}
= \jmath^{\ast}\pi^{\ast}).
\end{eqnarray*}
By the projection formula for $\varpi$,
it follows
$$
\Omega^{p}_{Z}(\imath^{\ast}W)
=\Omega_{Z}^{p}\otimes \imath^{\ast}W
\stackrel{\simeq}\longrightarrow
\varpi_{\ast}\Omega_{E}^{p}\otimes \imath^{\ast}W
\stackrel{\simeq}\longrightarrow
\varpi_{\ast}(\Omega_{E}^{p} \otimes \varpi^{\ast}\imath^{\ast}W).
$$

Finally, we give the proof of (iii).
Because $E$ is the exceptional divisor,
there exist two short exact sequences of vector bundles, i.e., the {\it structure sheaf sequence} and the dual of the {\it normal bundle sequence} on $\tilde{X}$ and $E$, respectively,
\begin{equation}\label{structsf-seq}
\xymatrix@C=0.5cm{
  0 \ar[r] &  \mathcal{O}_{\tilde{X}}(-E)  \ar[r] &  \mathcal{O}_{\tilde{X}} \ar[r] & \jmath_{\ast}\mathcal{O}_{E}  \ar[r] & 0; }
\end{equation}
\begin{equation}\label{normalbd-seq}
\xymatrix@C=0.5cm{
  0 \ar[r] & \mathcal{O}_{E}(-E) \ar[r] & \jmath^{\ast}\Omega_{\tilde{X}} \ar[r] & \Omega_{E} \ar[r] & 0, }
\end{equation}
where $\mathcal{O}_{E}(-E)=\jmath^{\ast}\mathcal{O}_{\tilde{X}}(-E)$.
Set $\mathcal{O}_{\tilde{X}}(1):=\mathcal{O}_{\tilde{X}}(-E)$,
$\mathcal{O}_{E}(1):=\mathcal{O}_{E}(-E)$
 and thus we have $\mathcal{O}_{E}(1)=\jmath^{\ast}\mathcal{O}_{\tilde{X}}(1)$.

For any integer $1\leq p\leq n-1$,
by tensoring the sequence \eqref{structsf-seq} with
$
\Omega^{p}_{\tilde{X}}(\tilde{W})\otimes \mathcal{O}_{\tilde{X}}(m)
$
and using the projection formula for the last term,
we obtain a short exact sequence of holomorphic vector bundles over $\tilde{X}$
\begin{equation}\label{twist-structsf-seq}
\xymatrix@C=0.5cm{
  0 \ar[r] & \Omega^{p}_{\tilde{X}}(\tilde{W})
  \otimes \mathcal{O}_{\tilde{X}}(m+1)  \ar[r] &  \Omega^{p}_{\tilde{X}}(\tilde{W})\otimes
  \mathcal{O}_{\tilde{X}}(m) \ar[r] &
  \Omega^{p}_{\tilde{X}}(\tilde{W})
  \otimes \jmath_{\ast} \mathcal{O}_{E}(m)  \ar[r] & 0. }
\end{equation}
Taking $p$-th exterior wedge of \eqref{normalbd-seq},
we obtain an exact sequence
\begin{equation}\label{exact-1.0}
\xymatrix@C=0.5cm{
  0 \ar[r] &  \Omega_{E}^{p-1}\otimes \mathcal{O}_{E}(1) \ar[r] & \jmath^{\ast}\Omega_{\tilde{X}}^{p}  \ar[r] & \Omega_{E}^{p}\ar[r] & 0. }
\end{equation}
Tensoring $\jmath^{\ast}\tilde{W}\otimes\mathcal{O}_{E}(m)$ with \eqref{exact-1.0},
we get another short exact sequence of vector bundles
\begin{equation}\label{twist-normalbd-seq}
\xymatrix@C=0.5cm{
  0 \ar[r] &  \Omega^{p-1}_{E}(\jmath^{\ast}\tilde{W})
  \otimes \mathcal{O}_{E}(m+1) \ar[r] & \jmath^{\ast}\Omega^{p}_{\tilde{X}}(\tilde{W})
  \otimes \mathcal{O}_{E}(m)  \ar[r] & \Omega^{p}_{E}(\jmath^{\ast}\tilde{W})
  \otimes  \mathcal{O}_{E}(m)\ar[r] & 0. }
\end{equation}

From now on, we fix an integer $q\geq 1$ and consider the higher direct images of the first and last terms in \eqref{twist-normalbd-seq}.
Observe that the line bundle
$$
\mathcal{O}_{\tilde X}(1)|_{E}=\mathcal{O}_{\tilde X}(-E)|_E\cong \mathcal{O}_{\mathbb{P}(N_{Z/X})}(1)
$$
is positive \cite[Lemma 3.26]{Voi02}.
Set $\mathcal{F}^{k}(l)=\Omega_{E}^{k}\otimes \mathcal{O}_{E}(l)$
and the Bott vanishing theorem \cite[Theorem 5.2]{Kod05}
implies for any $l\geq 1$ and $z\in Z$,
$$
H^{q}(E_{z}, \mathcal{F}^{k}(l)_{z})\cong
H^{q}(\mathbb{P}^{r-1}, \Omega_{\mathbb{P}^{r-1}}^{k}\otimes \mathcal{O}_{\mathbb{P}^{r-1}}(l))=0.
$$
So we get
$R^{q}\varpi_{\ast}(\Omega_{E}^{k}\otimes \mathcal{O}_{E}(l))=0$ for any integer $l\geq1$.
Actually, from the Grauert continuity theorem \cite[Theorem 4.12 (ii) of Chapter III]{BS76},
we know that
$R^{q}\varpi_{\ast}\mathcal{F}^{k}(l)$ is a locally free sheaf of rank $\mathrm{dim}\,H^{q}\bigl(E_{z}, \mathcal{F}^{k}(l)_{z}\bigr)=0$.
Moreover, there hold
\begin{equation}\label{bott-zero-1}
R^{q}\varpi_{\ast}\bigl
(\Omega^{p-1}_{E}(\jmath^{\ast}\tilde{W})\otimes \mathcal{O}_{E}(m+1)\bigr)
=
\mathcal{O}(\imath^{\ast}W)\otimes R^{q}\varpi_{\ast}\bigl
(\Omega_{E}^{p-1}\otimes \mathcal{O}_{E}(m+1)\bigr)
=0,\;\;\, \mathrm{for\,\,\,any}\,\,m\geq 0
\end{equation}
and similarly
\begin{equation*}\label{bott-zero-2}
R^{q}\varpi_{\ast}\bigl
(\Omega^{p}_{E}(\jmath^{\ast}\tilde{W})
\otimes  \mathcal{O}_{E}(m)\bigr)=0,\;\;\,
 \mathrm{for\,\,\,any}\,\,m\geq 1.
\end{equation*}
Hence,
by the long exact sequence of higher direct images of \eqref{twist-normalbd-seq},
we obtain
\begin{equation}\label{vansh1}
R^{q}\varpi_{\ast}\bigl
(\jmath^{\ast}\Omega^{p}_{\tilde{X}}(\tilde{W})\otimes \mathcal{O}_{E}(m)\bigr)=0,\;\;\,
 \mathrm{for\,\,\,any}\,\,m\geq 1.
\end{equation}
If $m=0$, then \eqref{twist-normalbd-seq} becomes
\begin{equation}\label{twist-normalbd-seq-m=0}
\xymatrix@C=0.5cm{
  0 \ar[r] &  \Omega^{p-1}_{E}(\jmath^{\ast}\tilde{W})\otimes \mathcal{O}_{E}(1) \ar[r] & \jmath^{\ast}\Omega^{p}_{\tilde{X}}(\tilde{W})  \ar[r] & \Omega^{p}_{E}(\jmath^{\ast}\tilde{W})\ar[r] & 0. }
\end{equation}
Likewise, using the vanishing result \eqref{bott-zero-1} of the case $m=0$, one sees that the exactness of the long exact sequence of higher direct images associated to \eqref{twist-normalbd-seq-m=0} implies the following isomorphism
\begin{equation}\label{import1}
R^{q}\varpi_{\ast}\jmath^{\ast}
\Omega_{\tilde{X}}^{p}(\tilde{W})
\stackrel{\simeq}\longrightarrow
R^{q}\varpi_{\ast}
\Omega^{p}_{E}(\jmath^{\ast}\tilde{W}),\;\;\,
\mathrm{for\,\,any}\,\,q\geq 1.
\end{equation}

Next we construct an isomorphism between the higher direct images of the third term in \eqref{twist-structsf-seq} and the second term in \eqref{twist-normalbd-seq}.
Based on the exactness of the functor $\jmath_{\ast}$ and $\imath_{\ast}$
we have
\begin{eqnarray*}
R^{q}\pi_{\ast}\bigl(\Omega^{p}_{\tilde{X}}(\tilde{W})
\otimes \jmath_{\ast} \mathcal{O}_{E}(m)\bigr)
&\stackrel{\simeq}\longrightarrow&  R^{q}\pi_{\ast}\bigl(\jmath_{\ast} (\jmath^{\ast} \Omega^{p}_{\tilde{X}}(\tilde{W})
\otimes \mathcal{O}_{E}(m))\bigr)
\;\;\;\;\; (\textrm{by projection formula}) \\
&\cong& R^{q}(\pi \circ \jmath)_{\ast} \bigl(\jmath^{\ast}\Omega^{p}_{\tilde{X}}(\tilde{W})
\otimes \mathcal{O}_{E}(m)\bigr)
\;\;\; (\jmath\,\,\textrm{is a closed embedding}) \\
&\cong& R^{q}(\imath \circ \varpi)_{\ast} \bigl(\jmath^{\ast}\Omega^{p}_{\tilde{X}}(\tilde{W})
\otimes \mathcal{O}_{E}(m)\bigr)
\;\; \,(\pi\circ \jmath=\imath \circ \varpi) \\
&\cong&\imath_{\ast}R^{q}\varpi_{\ast} \bigl(\jmath^{\ast}\Omega^{p}_{\tilde{X}}(\tilde{W})
\otimes \mathcal{O}_{E}(m)\bigr)\;\;
\;\;\;\;\;(\imath\,\,\textrm{is a closed embedding}).
\end{eqnarray*}
In summary, for $m\geq0$ there exists an isomorphism
\begin{equation}\label{i-iso}
R^{q}\pi_{\ast}\bigl(\Omega^{p}_{\tilde{X}}(\tilde{W}) \otimes \jmath_{\ast} \mathcal{O}_{E}(m)\bigr)
\stackrel{\simeq}\longrightarrow
\imath_{\ast}R^{q}\varpi_{\ast} \bigl(\jmath^{\ast}\Omega^{p}_{\tilde{X}}(\tilde{W})
\otimes \mathcal{O}_{E}(m)\bigr).
\end{equation}
Therefore, if $m\geq 1$, then combining \eqref{vansh1} with
\eqref{i-iso}, we get
\begin{equation}\label{vansh1.1}
R^{q}\pi_{\ast}\bigl(\Omega^{p}_{\tilde{X}}(\tilde{W})
\otimes \jmath_{\ast} \mathcal{O}_{E}(m)\bigr)=0.
\end{equation}
In the case of $m=0$, \eqref{import1} and \eqref{i-iso} imply
\begin{equation*}\label{import2}
R^{q}\pi_{\ast}(\Omega_{\tilde{X}}^{p}(\tilde{W}) \otimes \jmath_{\ast} \mathcal{O}_{E})
\stackrel{\simeq}\longrightarrow
\imath_{\ast}R^{q}\varpi_{\ast} \jmath^{\ast}\Omega_{\tilde{X}}^{p}(\tilde{W})
\stackrel{\simeq}\longrightarrow
\imath_{\ast}R^{q}\varpi_{\ast} \Omega_{E}^{p}(\jmath^{*}\tilde{W}).
\end{equation*}
Hence, to prove that the assertion holds, i.e.,
$$
R^{q}\pi_{\ast} \Omega^{p}_{\tilde{X}}(\tilde{W})
\stackrel{\simeq}\longrightarrow
\imath_{\ast}R^{q}\varpi_{\ast}
\Omega^{p}_{E}(\jmath^{\ast}\tilde{W}),
$$
it suffices to show
$
R^{q}\pi_{\ast}\Omega^{p}_{\tilde{X}}(\tilde{W})
\stackrel{\simeq}\longrightarrow
R^{q}\pi_{\ast}(\Omega_{\tilde{X}}^{p}(\tilde{W}) \otimes \jmath_{\ast} \mathcal{O}_{E}).
$

Let us turn back to the long exact sequence of higher direct images associated to \eqref{twist-structsf-seq}.
On one hand, since $X$ is compact and the blow-up morphism $\pi$ is projective (cf. \cite[Remarks 2.1 (0)]{GPR94}), the projection formula and the Grauert-Remmert theorem \cite[Theorem 2.1 (B) of  Chapter IV]{BS76} give
\begin{equation}\label{serre-zero}
R^{q}\pi_{\ast}\bigl
(\Omega_{\tilde{X}}^{p}(\tilde{W})\otimes \mathcal{O}_{\tilde{X}}(l)\bigr)
=\mathcal{O}(W)\otimes R^{q}\pi_{\ast}\bigl
(\Omega_{\tilde{X}}^{p}\otimes \mathcal{O}_{\tilde{X}}(l)\bigr)
=0
\end{equation}
for any $l\geq l_0$ with some integer $l_0=l_0(X,\Omega_{\tilde{X}}^{p})$.
On the other hand, the vanishing result \eqref{vansh1.1} implies that the morphism
\begin{equation}\label{vanishsurj}
R^{q}\pi_{\ast}\bigl
(\Omega_{\tilde{X}}^{p}(\tilde{W})\otimes \mathcal{O}_{\tilde{X}}(m+1)\bigr)
\rightarrow
R^{q}\pi_{\ast}\bigl
(\Omega_{\tilde{X}}^{p}(\tilde{W})\otimes \mathcal{O}_{\tilde{X}}(m)\bigr)
\end{equation}
is surjective for any $m\geq 1$.
From \eqref{serre-zero} and \eqref{vanishsurj} we get
\begin{equation}\label{last-zero}
R^{q}\pi_{\ast}\bigl
(\Omega_{\tilde{X}}^{p}(\tilde{W})\otimes \mathcal{O}_{\tilde{X}}(1)\bigr)=0,\;\,\,
\mathrm{for\,\,any}\,\,\,q\geq 1
\end{equation}
by an induction for the index $m$.
Let $m=0$ in the long exact sequence of higher direct images associated to \eqref{twist-structsf-seq}.
On account of the vanishing result \eqref{last-zero} and the exactness we obtain the desired isomorphism
$$
R^{q}\pi_{\ast}\Omega^{p}_{\tilde{X}}(\tilde{W})
\stackrel{\simeq}\longrightarrow R^{q}\pi_{\ast}(\Omega_{\tilde{X}}^{p}(\tilde{W})\otimes \jmath_{\ast} \mathcal{O}_{E})
$$
and this completes the proof.
\end{proof}
\subsection{Relative Dolbeault sheaves}
Let $X$ be a compact complex manifold and let
$\imath: Z\hookrightarrow X$
be a closed complex submanifold.
There exist two natural morphisms of sheaves:
$\imath^{\ast}:\Omega_{X}^{p}\rightarrow \imath_{\ast}\Omega^{p}_{Z}$ and $\imath^{\ast}:\A_{X}^{p, q} \rightarrow \imath_{\ast}\A_{Z}^{p, q}$
which are induced by the pullback as follows.
For any open subset $V\subset X$, we define the morphisms
\begin{eqnarray}\label{hol-surjmap}
  \imath^{\ast}(V):\Gamma(V,\Omega_{X}^{p})&\rightarrow&
  \Gamma(V,\imath_{\ast}\Omega^{p}_{Z})=\Gamma(V\cap Z, \Omega^{p}_{Z}) \nonumber \\
  \alpha &\mapsto&  (\imath_{V\cap Z})^{\ast} \alpha,
\end{eqnarray}
and
\begin{eqnarray}\label{smooth-surjmap}
\imath^{\ast}(V):\Gamma(V, \A_{X}^{p, q})&\rightarrow& \Gamma(V,\imath_{\ast}\A_{Z}^{p, q} )=\Gamma(V\cap Z, \A^{p,q}_{Z}) \nonumber \\
\alpha &\mapsto& (\imath_{V\cap Z})^{\ast} \alpha,
\end{eqnarray}
where $\imath_{V\cap Z}:V\cap Z\rightarrow V$ is the holomorphic inclusion.
It is direct to check that the maps $\imath^{\ast}(V)$ in \eqref{hol-surjmap} and \eqref{smooth-surjmap} are homomorphisms of $\mathcal{O}_{X}(V)$-module and  $\mathcal{C}_{X}^{\infty}(V)$, respectively.
Hence, $\imath^{\ast}:\Omega_{X}^{p}\rightarrow \imath_{\ast}\Omega^{p}_{Z}$ is a morphism of $\mathcal{O}_{X}$-modules and $\imath^{\ast}:\A_{X}^{p, q} \rightarrow \imath_{\ast}\A_{Z}^{p, q}$
is a morphism of $\mathcal{C}_{X}^{\infty}$-modules.
For any $0\leq p,q\leq n$, consider the kernel sheaves
$
\K_{X,Z}^{p}:=\ker\big(\Omega_{X}^{p} \stackrel{\imath^{\ast}}\longrightarrow  \imath_{\ast}\Omega_{Z}^{p}\big)
$
and
$
\K_{X,Z}^{p,q}:=\ker\big(\A_{X}^{p,q} \stackrel{\imath^{\ast}}\longrightarrow  \imath_{\ast}\A_{Z}^{p,q}\big).
$
\begin{defn}\label{rds}
We say that $\K_{X,Z}^{p}$ and $\K_{X,Z}^{p,q}$ are the {\it $p$-th and $(p,q)$-th relative Dolbeault sheaves} of $X$ with respect to $Z$, respectively.
\end{defn}

From the definitions, $\K_{X,Z}^{p}$ is a coherent sheaf of $\mathcal{O}_{X}$-module
and $\K_{X,Z}^{p,q}$ is a sheaf of $\mathcal{C}_{X}^{\infty}$-module.
If $p=0$, then $\K_{X,Z}^{0}$ is the ideal sheaf $\mathcal{I}_{Z}$ of $Z$ in $X$;
if $p=n$, then $\K_{X,Z}^{n}=\Omega_{X}^{n}$.
Moreover, we have the following result.
\begin{prop}
\label{redol}
With the above setting, there hold
\begin{enumerate}
\item[(i)] $\K_{X,Z}^{p,\bullet}$ is a fine resolution of $\K_{X,Z}^{p}$;
\item[(ii)] $\imath^{\ast}:\Omega_{X}^{p}\rightarrow \imath_{\ast}\Omega_{Z}^{p}$ is a surjective morphism of $\mathcal{O}_{X}$-modules;
\item[(iii)] $\imath^{\ast}:\A_{X}^{p,q}\rightarrow \imath_{\ast}\A_{Z}^{p,q}$ is a surjective morphism of $\mathcal{C}_{X}^{\infty}$-modules.
\end{enumerate}
\end{prop}

\begin{proof}
(i)
By the Dolbeault-Grothendieck lemma \cite[Proposition 2.31]{Voi02}, the sheaf complex $\K_{X,Z}^{p,\bullet}$ is exact, and hence $\K_{X,Z}^{p,\bullet}$ is a fine resolution of $\K_{X,Z}^{p}$.

(ii)
It suffices to verify that for any $x\in Z$ the morphism of stalks
$
\imath^{\ast}_{x}:(\Omega^{p}_{X})_{x}
\rightarrow(\imath_{\ast}\Omega^{p}_{Z})_{x}
$
is surjective.
Let $(\mathcal{U}; z_{1},z_{2},\cdots,z_{n})$ be a local coordinate chart of $x$ such that
$$\mathcal{U}\cap Z=\{z_{n-r+1}=\cdots=z_{n}=0\}.$$
Then there exists a holomorphic projection given by
\begin{equation*}
\tau:\mathcal{U} \rightarrow \mathcal{U}\cap Z,\,\,\,
(z_{1},\cdots,z_{n}) \mapsto (z_{1},\cdots,z_{n-r}).
\end{equation*}
For any germ $\alpha_{x} \in (\imath_{\ast}\Omega_{Z}^{p})_{x}$ we can choose a representative $(V,\alpha)$,
where $\alpha$ is a holomorphic $p$-form on $V\cap Z$.
Particularly, we can choose $V$ small enough such that $V\subset \mathcal{U}$ and then the restriction of $\tau$ on $V$ gives rise to a holomorphic map
$\tau_{V}:V\rightarrow V\cap Z$
such that
$\tau_{V}\circ \iota_{V\cap Z}=\textmd{id}_{V\cap Z}$.
Let $\beta=(\tau_{V})^{\ast}(\alpha)$. Then $(V,\beta)$ represents a germ, denoted by $\beta_{x}$, in the stalk $(\Omega^{p}_{X})_{x}$.
From definition, we get
$$
(\imath_{V\cap Z})^{\ast}\beta=(\imath_{V\cap Z})^{\ast}\bigl((\tau_{V})^{\ast}(\alpha)\bigr)
=(\tau_{V}\circ\imath_{V\cap Z})^{\ast}(\alpha)=\alpha.
$$
It follows that $\imath^{\ast}_{x}(\beta_{x})=\alpha_{x}$, namely, $\imath^{\ast}_{x}$ is surjective.

The proof of (iii) is similar to (ii); see also \cite[Lemma 3.9]{RYY}$_{v3}$.
\end{proof}

By Proposition \ref{redol} (i),
we get the isomorphisms
$$
H^{q}(X, \K_{X,Z}^{p})\cong \mathbb{H}^{q}(X, \K_{X,Z}^{p,\bullet})\cong H^{p,q}(X,Z),
$$
where  $H^{p,q}(X,Z)$ is the relative Dolbeault cohomology of $X$ with respect to $Z$ defined in \cite{RYY}$_{v4}$.
Moreover, the assertion Proposition \ref{redol} (ii) yields a short exact sequence of $\mathcal{O}_{X}$-modules
\begin{equation}\label{rel-Dol-sequ-holom}
\xymatrix@C=0.5cm{
  0 \ar[r] & \K^{p}_{X,Z} \ar[r]^{} & \Omega_{X}^{p} \ar[r]^{\imath^{\ast}} & \imath_{\ast}\Omega_{Z}^{p} \ar[r] & 0,}
\end{equation}
and the assertion Proposition \ref{redol} (iii)
induces a short exact sequence of $\mathcal{C}^{\infty}_{X}$-modules:
\begin{equation}\label{rel-Dol-sequ-smooth}
\xymatrix@C=0.5cm{
  0 \ar[r] & \K^{p,q}_{X,Z} \ar[r]^{} & \A_{X}^{p,q} \ar[r]^{\imath^{\ast}} & \imath_{\ast}\A_{Z}^{p,q} \ar[r] & 0.}
\end{equation}

As $W$ is a holomorphic vector bundle over $X$, the sheaf of holomorphic sections of $W$ is a sheaf of $\mathcal{O}_{X}$-modules.
Consequently, the tensor functor $-\otimes_{\mathcal{O}_{X}}\mathcal{O}(W)$ is exact.
Tensoring \eqref{rel-Dol-sequ-holom} with $\mathcal{O}(W)$,
one gets the short sequence of $\mathcal{O}_{X}$-modules
\begin{equation}\label{Z-holomor-seq}
\xymatrix@C=0.5cm{
  0 \ar[r] & \mathscr{K}^{p}_{X,Z} \otimes_{\mathcal{O}_{X}} \mathcal{O}(W) \ar[r]^{} & \Omega_{X}^{p} \otimes_{\mathcal{O}_{X}} \mathcal{O}(W)\ar[r]^{\imath^{\ast}} & \imath_{\ast} \Omega_{Z}^{p}\otimes_{\mathcal{O}_{X}} \mathcal{O}(W) \ar[r] & 0}
\end{equation}
is exact.
Note that
$
\Omega^{p}_{X}(W)=\Omega_{X}^{p} \otimes_{\mathcal{O}_{X}} W\cong\mathcal{O}(\wedge^{p}T^{\prime}X\otimes_{\mathbb{C}}W).
$
From the projection formula, we have
\begin{eqnarray*}
\imath_{\ast} \Omega_{Z}^{p} \otimes_{\mathcal{O}_{X}} \mathcal{O}(W)
&\stackrel{\simeq}\longrightarrow&
\imath_{\ast}( \Omega_{Z}^{p} \otimes_{\mathcal{O}_{Z}} \imath^{\ast} \mathcal{O}(W))\\
&\cong&
\imath_{\ast}( \Omega_{Z}^{p} \otimes_{\mathcal{O}_{Z}}\mathcal{O}(\imath^{\ast}W) )\\
&=&\imath_{\ast}\Omega^{p}_{Z}(\imath^{\ast}W).
\end{eqnarray*}
Set $\mathscr{K}^{p}_{X,Z}(W)=\mathscr{K}^{p}_{X,Z} \otimes_{\mathcal{O}_{X}} \mathcal{O}(W)$.
Then the short exact sequence \eqref{Z-holomor-seq} is equal to
\begin{equation}\label{W-sheaf-exact-seq}
\xymatrix@C=0.5cm{
  0 \ar[r] & \mathscr{K}^{p}_{X,Z}(W) \ar[r]^{} & \Omega^{p}_{X}(W) \ar[r]^{\imath^{\ast}} & \imath_{\ast}\Omega^{p}_{Z}(\imath^{\ast}W) \ar[r] & 0. }
\end{equation}
Akin to \eqref{W-sheaf-exact-seq}, for the pair of the holomorphic vector bundles $(\tilde{W},\jmath^{\ast}\tilde{W})$ we have a short exact sequence of sheaves on $\tilde{X}$ as follows
\begin{equation}\label{tilde-W-sheaf-exact-seq}
\xymatrix@C=0.5cm{
  0 \ar[r] & \mathscr{K}^{p}_{\tilde{X},E}(\tilde{W}) \ar[r]^{} & \Omega^{p}_{\tilde{X}}(\tilde{W}) \ar[r]^{\jmath^{\ast}} & \jmath_{\ast}\Omega^{p}_{E}(\jmath^{\ast}\tilde{W}) \ar[r] & 0. }
\end{equation}

\begin{lem}\label{hi-dir-0}
For any $0\leq p\leq n$,
the pullback of holomorphic forms induces an isomorphism
$
\pi^{\ast}:
\mathscr{K}^{p}_{X,Z}(W)
\stackrel{\simeq}\longrightarrow
\pi_{\ast}\mathscr{K}^{p}_{\tilde{X},E}(\tilde{W});
$
moreover, for any $q\geq1$, we have
$
R^{q}\pi_{\ast}\mathscr{K}^{p}_{\tilde{X},E}(\tilde{W})=0.
$
\end{lem}

\begin{proof}
Since
$\pi\circ\jmath=\imath\circ\varpi$ \eqref{blowup-diag3},
we have
$$
\pi_{\ast}\jmath_{\ast}\Omega^{p}_{E}(\jmath^{\ast}\tilde{W})
=
\imath_{\ast}\varpi_{\ast}\Omega^{p}_{E}(\jmath^{\ast}\tilde{W}).
$$
From (i) and (ii) in Lemma \ref{key-lem-1},
there is a commutative diagram of exact sequences
\begin{equation}\label{kershisodiag}
\xymatrix@C=0.5cm{
0 \ar[r]^{} & \K_{X, Z}^{p}(W)
\ar[d]_{\pi^{\ast}}^{} \ar[r]^{} &
\Omega_{X}^{p}(W)
\ar[d]_{\pi^{\ast}}^{\cong} \ar[r]^{\imath^{\ast}} &
\imath_{\ast}\Omega_{Z}^{p}(\imath^{\ast}W)
\ar[d]_{{\varpi}^{\ast}}^{\cong} \ar[r]^{} & 0 \\
0 \ar[r] &
\pi_{\ast}\K_{\tilde{X}, E}^{p}(\tilde{W})
\ar[r]^{} &
\pi_{\ast}\Omega_{\tilde{X}}^{p}(\tilde{W})
\ar[r]^{\jmath^{\ast}} &
\pi_{\ast}\jmath_{\ast}\Omega_{E}^{p}(\jmath^{\ast}\tilde{W})
.}
\end{equation}
Observe that
$
\imath^{\ast}:\Omega^{p}_{X}(W)
\rightarrow
\imath_{\ast}\Omega^{p}_{Z}(\imath^{\ast}W)
$
is surjective by \eqref{W-sheaf-exact-seq}.
Thus the sheaf morphism
$
\jmath^{\ast}:
\pi_{\ast}\Omega^{p}_{\tilde{X}}(\tilde{W})\rightarrow
\pi_{\ast}\jmath_{\ast}\Omega^{p}_{E}(\jmath^{\ast}\tilde{W})
$
is surjective by the commutative diagram \eqref{kershisodiag}.
As a result, the morphism
$
\pi^{\ast}: \mathscr{K}^{p}_{X,Z}(W)\rightarrow \pi_{\ast}\mathscr{K}^{p}_{\tilde{X},E}(\tilde{W})
$
is isomorphic.

Now we consider the higher direct images of $\mathscr{K}_{\tilde{X},E}(W)$ along $\pi$.
The short exact sequence \eqref{tilde-W-sheaf-exact-seq}
induces a long exact sequence: for $q\geq0$,
$$
\xymatrix@C=0.5cm{
  \cdots \ar[r] &
  R^{q}\pi_{\ast}\mathscr{K}^{p}_{\tilde{X},E}(\tilde{W}) \ar[r]^{}
  & R^{q}\pi_{\ast}\Omega^{p}_{\tilde{X}}(\tilde{W}) \ar[r]^{\jmath^{\ast}}
  & R^{q}\pi_{\ast}
  \jmath_{\ast}\Omega^{p}_{E}(\jmath^{\ast}\tilde{W}) \ar[r]^{}
  & R^{q+1}\pi_{\ast}\mathscr{K}^{p}_{\tilde{X},E}(\tilde{W}) \ar[r]^{}
  & \cdots.}
$$

To prove that the higher direct images vanish,
we need the following isomorphisms
\begin{eqnarray}\label{import-isos-on-E}
R^{q}\pi_{\ast}\jmath_{\ast}\Omega^{p}_{E}(\jmath^{\ast}\tilde{W})
&\cong& R^{q}(\pi\circ\jmath)_{\ast}\Omega^{p}_{E}(\jmath^{\ast}\tilde{W}) \nonumber \\
&\cong& R^{q}(\imath\circ\varpi)_{\ast}\Omega^{p}_{E}(\jmath^{\ast}\tilde{W})\nonumber \\
&\cong& \imath_{\ast}R^{q}\varpi_{\ast}\Omega^{p}_{E}(\jmath^{\ast}\tilde{W}).
\end{eqnarray}
Combining Lemma \ref{key-lem-1} (iii) with the above equality we get
$$
\xymatrix@C=0.5cm{
  \cdots \ar[r] &
  R^{q}\pi_{\ast}\mathscr{K}^{p}_{\tilde{X},E}(\tilde{W}) \ar[r]^{}
  & R^{q}\pi_{\ast}\Omega^{p}_{\tilde{X}}(\tilde{W}) \ar[r]^{\jmath^{\ast}\;\;\;}_{\cong\;\;\;}
  & \imath_{*}R^{q}\varpi_{\ast}
  \Omega^{p}_{E}(\jmath^{\ast}\tilde{W}) \ar[r]^{}
  & R^{q+1}\pi_{\ast}\mathscr{K}^{p}_{\tilde{X},E}(\tilde{W}) \ar[r]^{}
  & \cdots}
$$
with $q\geq1$.
The exactness of the above sequence induces
$
R^{q}\pi_{\ast}\mathscr{K}^{p}_{\tilde{X},E}(\tilde{W})=0,
$
for any $q\geq 2$.

It remains to show that $R^{1}\pi_{\ast}\mathscr{K}^{p}_{\tilde{X},E}(\tilde{W})$
vanishes.
Consider the exact sequence
$$
\xymatrix@C=0.4cm{
\pi_{\ast}\Omega^{p}_{\tilde{X}}(\tilde{W})
  \ar[r]^{\jmath^{\ast}} & \pi_{\ast}\jmath_{\ast}\Omega^{p}_{E}(\jmath^{\ast}\tilde{W})
  \ar[r]^{\rho} &
  R^{1}\pi_{\ast}\mathscr{K}^{p}_{\tilde{X},E}(\tilde{W})
  \ar[r]^{\kappa} & R^{1}\pi_{\ast}\Omega^{p}_{\tilde{X}}(\tilde{W})
  \ar[r]^{\cong\quad} & \imath_{\ast}R^{1}\varpi_{\ast}\Omega^{p}_{E}
  (\jmath^{\ast}\tilde{W}). &  }
$$
The exactness means $\mathrm{Im}\,(\kappa)=0$ and
$\ker\,(\kappa)=
R^{1}\pi_{\ast}\mathscr{K}^{p}_{\tilde{X},E}(\tilde{W})
=\mathrm{Im}\,(\rho)$, i.e.,
$\rho$ is surjective.
So we get
\begin{equation*}
  R^{1}\pi_{\ast}\mathscr{K}^{p}_{\tilde{X},E}(\tilde{W})
  \cong \pi_{\ast}\jmath_{\ast}
  \Omega^{p}_{E}(\jmath^{\ast}\tilde{W})/
  \ker\,(\rho)
  = \pi_{\ast}\jmath_{\ast}
  \Omega^{p}_{E}(\jmath^{\ast}\tilde{W})/
  \mathrm{Im}\,(\jmath^{\ast})\,\,\,
  = 0
\end{equation*}
because $\jmath^{*}$ is surjective.

This lemma also can be proved by \cite[Lemma 2.4]{Men18} and the projection formula.
\end{proof}

We are ready to present the most important result of this section which plays a significant role in the proof of Theorem \ref{thm1}.
\begin{lem}\label{kersheaf-iso}
For any $q\in \mathbb{N}$, the pullback $\pi^{\ast}$ induces an isomorphism as abelian groups
\begin{equation}\label{crd-iso}
\pi^{\ast}: H^{q}(X, \K_{X,Z}^{p}(W)) \stackrel{\simeq}\longrightarrow H^{q}(\tilde{X}, \K_{\tilde{X},E}^{p}(\tilde{W})).
\end{equation}
\end{lem}

\begin{proof}
Pick $\K_{\tilde{X},E}^{p}(\tilde{W})\rightarrow \mathcal{G}^{p, \bullet}$ a flabby resolution.
By the definition of sheaf cohomology, we have
$$
H^{l}(\tilde{X}, \K_{\tilde{X},E}^{p}(\tilde{W}))
:=
H^{l}(\Gamma(\tilde{X}, \mathcal{G}^{p, \bullet})).
$$
As defined in \cite[(13.4) Definition of Chapter IV]{Dem12}, the higher direct images are
$$
R^{q}\pi_{\ast}\K_{\tilde{X},E}^{p}(\tilde{W})
:=
\mathscr{H}^{q}(\pi_{\ast}\mathcal{G}^{p, \bullet}).
$$
According to Lemma \ref{hi-dir-0}, $R^{q}\pi_{\ast}\K_{\tilde{X},E}^{p}(\tilde{W})=0$ for $q\neq 0$, hence $\mathscr{H}^{q}(\pi_{\ast}\mathcal{G}^{p, \bullet})=0$ for $q\neq 0$, which means that the sheaf complex $\pi_{\ast}\mathcal{G}^{p, \bullet}$ is exact.
It concludes that $\pi_{\ast}\mathcal{G}^{p, \bullet}$ is a flabby resolution of $\pi_{\ast}\K_{\tilde{X},E}^{p}(\tilde{W})$ since the direct image of a flabby sheaf is flabby.
Therefore, we derive an isomorphism of abelian groups
\begin{equation}\label{more-iso}
H^{l}(X, \pi_{\ast}\K_{\tilde{X},E}^{p}(\tilde{W}))
=
H^{l}(\Gamma(X, \pi_{\ast}\mathcal{G}^{p, \bullet}))
=
H^{l}(\Gamma(\tilde{X}, \mathcal{G}^{p, \bullet}))
=
H^{l}(\tilde{X}, \K_{\tilde{X},E}^{p}(\tilde{W})).
\end{equation}
Moreover, since $\pi^{\ast}: \K_{X,Z}^{p}(W)\rightarrow \pi_{\ast}\K_{\tilde{X},E}^{p}(\tilde{W})$ is an isomorphism as in Lemma \ref{hi-dir-0}, it induces an isomorphism
$$
\pi^{\ast}: H^{q}(X, \K_{X,Z}^{p}(W)) \stackrel{\simeq}\longrightarrow H^{q}(X, \pi_{\ast}\K_{\tilde{X},E}^{p}(\tilde{W})).
$$
Combining it with \eqref{more-iso} concludes the proof.
\end{proof}

\section{Proof of Theorem \ref{thm1}}\label{proof}
\subsection{Construction of the morphism $\phi$}
\label{def-phi}
Let $W^{\ast}$ be the dual bundle to the holomorphic vector bundle $W$ over $X$.
Then the pullback $\pi^{\ast}W^{\ast}$ is a holomorphic vector bundle on $\tilde{X}$.
Set $\mathscr{D}^{p,q}(X,W)$ as the space of $W$-valued currents of type $(p,q)$ on $X$,
which is defined to be the dual of the topological vector space $A^{n-q,n-q}(X,W^{\ast})$ equipped with its natural topology.
Moreover, the operator $\bar{\partial}$ induces a differential $\bar{\partial}^{\ast}$ on
$\mathscr{D}^{p,\bullet}(X,W)$.
The associated $q$-th cohomology of the dual complex
$\{\mathscr{D}^{p,\bullet}(X,W),\bar{\partial}^{\ast}\}$
is denoted by $H^{p,q}_{\mathscr{D}}(X,W)$.
From definition, there is a natural inclusion
$
\varrho:A^{p,\bullet}(X,W)
\hookrightarrow
\mathscr{D}^{p,\bullet}(X,W)
$
which induces an isomorphism
\begin{equation*}
\varrho_{\ast}:H^{p,q}(X,W)
\stackrel{\simeq}\longrightarrow
H^{p,q}_{\mathscr{D}}(X,W)
\end{equation*}
as shown in the proof of \cite[Theorem 3.3]{Wel74}.
Similarly, for the $\pi^{\ast}W$-valued currents on $\tilde{X}$ we have
\begin{equation*}
\tilde{\varrho}_{\ast}:H^{p,q}(\tilde{X},\pi^{*}W)
\stackrel{\simeq}\longrightarrow
H^{p,q}_{\mathscr{D}}(\tilde{X},\pi^{*}W).
\end{equation*}
The pushforward of the currents defines a morphism
$$
\pi_{\flat}:
H^{p,q}_{\mathscr{D}}(\tilde{X},\pi^{\ast}W)
\rightarrow
H^{p,q}_{\mathscr{D}}(X,W)
$$
and gives the following diagram
\begin{equation}\label{form-curr}
\xymatrix{
  H^{p,q}(\tilde{X},\pi^{\ast}W)  \ar[r]^{\tilde{\varrho}_{\ast}}_{\cong}
  & H^{p,q}_{\mathscr{D}}(\tilde{X},\pi^{\ast}W) \ar[d]^{\pi_{\flat}} \\
  H^{p,q}(X,W)\ar[u]_{\pi^{\ast}}
  \ar[r]^{\varrho_{\ast}}_{\cong}
  & H^{p,q}_{\mathscr{D}}(X,W).   }
\end{equation}

As the degree of the blow-up morphism $\pi$ is $1$,
we get
$
\varrho_{\ast}=
\pi_{\flat}\circ\tilde{\varrho}_{\ast}
\circ\pi^{\ast}
$
(cf. \cite[Lemma 2.2]{Wel74}).
For any $0\leq p,q\leq n$, the morphism
$$
\pi^{\ast}:H^{p,q}(X,W)
\rightarrow
H^{p,q}(\tilde{X},\pi^{\ast}W)
$$
is an injection (cf. \cite[Theorem 3.1 (c)]{Wel74}).
From \eqref{form-curr}, we can define a natural morphism
$$
\pi_{\ast}=
\varrho^{-1}_{\ast}\circ\pi_{\flat}
\circ\tilde{\varrho}_{\ast}:
H^{p,q}(\tilde{X},\pi^{\ast}W)\rightarrow H^{p,q}(X,W).
$$
Then \eqref{form-curr} implies that
$$
\pi_{\ast}:
H^{p,q}(\tilde{X},\pi^{\ast}W)\rightarrow H^{p,q}(X,W)
$$
is surjective.

In view of Lemma \ref{dolb-projective-formula}, we get
$$
H^{p,q}(E,\jmath^{\ast}\tilde{W})=
\bigoplus^{r-1}_{i=0}\bm{h}^{i}\wedge
\varpi^{\ast}H^{p-i,q-i}(Z,\imath^{\ast}W),
$$
where
$
\bm{h}=c_{1}(\mathcal{O}_{E}(1))\in H^{1,1}_{\bar{\partial}}(E)
$
and
$\varpi^{\ast}$ is the pullback of the projection
$\varpi:E\rightarrow Z$.
It means that each class
$[\tilde{\alpha}]_{(p,q)}\in H^{p,q}(E,\jmath^{\ast}\tilde{W})$
admits a unique expression
$$
[\tilde{\alpha}]_{(p,q)}=
\sum^{r-1}_{i=0}\bm{h}^{i}\wedge
\varpi^{\ast}[\alpha]_{(p-i,q-i)},
$$
where
$[\alpha]_{(p-i,q-i)}\in H^{p-i,q-i}(Z,\imath^{\ast}W)$.
Define the linear map
\begin{eqnarray*}
  \Pi_{i}:H^{p,q}(E,\jmath^{\ast}\tilde{W})
  &\rightarrow& H^{p-i,q-i}(Z,\imath^{\ast}W) \\
  {[\tilde{\alpha}]}_{(p,q)}&\mapsto& {[\alpha]}_{(p-i,q-i)}.
\end{eqnarray*}
Then we can define the desired morphism $\phi$ by setting
\begin{equation}\label{phi}
\phi=\pi_{\ast}+\sum_{i=1}^{r-1}\Pi_{i}\circ\jmath^{\ast}
\end{equation}
which maps
$\mathbb{V}^{p,q}:=H^{p,q}(\tilde{X},\pi^{\ast}W)$
to the space
$$
\mathbb{W}^{p,q}:=
H^{p,q}(X,W)\oplus
 \biggl(\bigoplus^{r-1}_{i=1}
 H^{p-i,q-i}(Z,\imath^{\ast}W)\biggr).
$$
\subsection{$\phi$ is isomorphic}
We divide the proof in two steps.
First, we prove that
$\mathbb{V}^{p,q}$ is isomorphic to $\mathbb{W}^{p,q}$ as complex vector spaces, i.e.,
they have the same complex dimensions.
Then we show that the linear morphism
$\phi:\mathbb{V}^{p,q}
\rightarrow
\mathbb{W}^{p,q}$
is injective.

In the proof we need the following basic result from homological algebra.
\begin{prop}\label{coker=coker}
Consider an exact ladder of complex vector spaces
\begin{equation*}
\xymatrix@C=0.5cm{
  \cdots \ar[r]^{} & A_1 \ar[d]_{i_1} \ar[r]^{f_1}& A_2 \ar[d]_{i_2} \ar[r]^{f_2}& A_3 \ar[d]_{i_3} \ar[r]^{f_3}& A_4 \ar[d]_{i_4} \ar[r]^{}& \cdots \\
\cdots \ar[r]^{} & B_1 \ar[r]^{g_1}& B_2 \ar[r]^{g_2}& B_3 \ar[r]^{g_3}& B_4 \ar[r]^{}&  \cdots.}
\end{equation*}
If the vertical maps $i_1,\,i_4$ are isomorphic, and $i_2,\,i_3$ are injective,
then $g_{2}$ induces an isomorphism
$$
B_2/i_{2}(A_2)\cong B_3/ i_3(A_3).
$$
\end{prop}

Tensoring \eqref{rel-Dol-sequ-smooth} with $W$ (viewed as a complex vector bundle),
we get a short exact sequence
\begin{equation*}
\xymatrix@C=0.5cm{
0 \ar[r]^{} & \K_{X,Z}^{p, q}(W) \ar[r]^{} & \A_{X}^{p, q}(W) \ar[r]^{\imath^{\ast}} & \imath_{\ast}(\A_{Z}^{p, q}(\imath^{\ast}W))\ar[r]^{} & 0}
\end{equation*}
for sheaves of $\mathcal{C}_{X}^{\infty}$-modules.
Note that $\K_{X,Z}^{p, q}(W)$ is a fine sheaf and hence $\Gamma$-acyclic (cf. \cite[Proposition 4.36]{Voi02}).
Moreover, since $\Gamma(X,-)$ is a left exact functor,
taking the global sections to the exact sequence above we obtain a short exact sequence of vector spaces
\begin{equation*}
\xymatrix@C=0.5cm{
0 \ar[r]^{} &  \Gamma(X,\K_{X,Z}^{p,q}(W))  \ar[r]^{} &\Gamma(X,\A_{X}^{p,q}(W)) \ar[r]^{\imath^{\ast}} & \Gamma(Z,\A_{Z}^{p,q}(\imath^{\ast}W))\ar[r]^{} & 0}.
\end{equation*}
Since $\imath^{\ast}$ commutes with the Dolbeault differential $\bar{\partial}$,
there exists a commutative diagram of short exact sequences
\begin{equation*}
\xymatrix@C=0.5cm{
 0 \ar[r]^{} & \Gamma(X,\K_{X,Z}^{p,q}(W)) \ar[d]_{\bar{\partial}} \ar[r]^{} & \Gamma(X,\A_{X}^{p,q}(W)) \ar[d]_{\bar{\partial}} \ar[r]^{\imath^{\ast}} &  \Gamma(Z,\A_{Z}^{p,q}(\imath^{\ast}W))\ar[d]_{\bar{\partial}} \ar[r]^{} & 0 \\
0 \ar[r] &  \Gamma(X,\K_{X,Z}^{p,q+1}(W)) \ar[r]^{} & \Gamma(X,\A_{X}^{p,q+1}(W)) \ar[r]^{\imath^{\ast}} & \Gamma(Z,\A_{Z}^{p,q+1}(\imath^{\ast}W)) \ar[r] &0. }
\end{equation*}
Therefore, there is a short exact sequence for complexes of vector spaces
\begin{equation*}\label{}
\xymatrix@C=0.5cm{
0 \ar[r]^{} &  \Gamma(X,\K_{X,Z}^{p,\bullet}(W))  \ar[r]^{} &\Gamma(X,\A_{X}^{p,\bullet}(W)) \ar[r]^{\imath^{\ast}} & \Gamma(Z,\A_{Z}^{p,\bullet}(\imath^{\ast}W))\ar[r]^{} & 0.}
\end{equation*}
Likewise, for the triple $(\tilde{X}, E, \tilde{W})$,
there exist a short exact sequence of $\mathcal{C}_{\tilde{X}}^{\infty}$-modules
\begin{equation*}\label{}
\xymatrix@C=0.5cm{
0 \ar[r]^{} & \K_{\tilde{X},E}^{p, q}(\tilde{W}) \ar[r]^{} & \A_{\tilde{X}}^{p, q}(\tilde{W}) \ar[r]^{\jmath^{\ast}} & \jmath_{\ast}(\A_{E}^{p, q}(\jmath^{\ast}\tilde{W}))\ar[r]^{} & 0}
\end{equation*}
and a short exact sequence for complexes of vector spaces
\begin{equation*}\label{}
\xymatrix@C=0.5cm{
0 \ar[r]^{} &  \Gamma(X,\K_{\tilde{X},E}^{p,\bullet}(\tilde{W}))  \ar[r]^{} &\Gamma(X,\A_{\tilde{X}}^{p,\bullet}(\tilde{W})) \ar[r]^{\jmath^{\ast}} & \Gamma(E,\A_{E}^{p,\bullet}(\jmath^{\ast}\tilde{W}))\ar[r]^{} & 0.}
\end{equation*}

By the blow-up diagram \eqref{blowup-diag3}, we have
$\jmath^{\ast}\circ\pi^{\ast}=\varpi^{\ast}\circ \imath^{\ast}$
and hence there is a commutative diagram for short exact sequences of vector spaces
\begin{equation*}
\xymatrix@C=0.5cm{
 0 \ar[r]^{} & \Gamma(X,\K_{X,Z}^{p,q}(W)) \ar[d]_{\pi^{\ast}} \ar[r]^{} & \Gamma(X,\A_{X}^{p,q}(W)) \ar[d]_{\pi^{\ast}} \ar[r]^{\imath^{\ast}} &  \Gamma(Z,\A_{Z}^{p,q}(\imath^{\ast}W))\ar[d]_{\varpi^{\ast}} \ar[r]^{} & 0 \\
0 \ar[r] &  \Gamma(\tilde{X},\K_{\tilde{X},E}^{p,q}(\tilde{W})) \ar[r]^{} &
   \Gamma(\tilde{X},\A_{\tilde{X}}^{p,q}(\tilde{W})) \ar[r]^{\jmath^{\ast}} &
   \Gamma(E,\A_{E}^{p,q}(\jmath^{\ast}\tilde{W})) \ar[r] &0. }
\end{equation*}
As the operator $\bar{\partial}$ commutes with the pullback maps,
there exists a commutative diagram
\begin{equation}\label{commut-complexes}
\xymatrix@C=0.5cm{
 0 \ar[r]^{} & \Gamma(X,\K_{X,Z}^{p,\bullet}(W)) \ar[d]_{\pi^{\ast}} \ar[r]^{} & \Gamma(X,\A_{X}^{p,\bullet}(W)) \ar[d]_{\pi^{\ast}} \ar[r]^{\imath^{\ast}} &  \Gamma(Z,\A_{Z}^{p,\bullet}(\imath^{\ast}W))\ar[d]_{\varpi^{\ast}} \ar[r]^{} & 0 \\
0 \ar[r] &  \Gamma(\tilde{X},\K_{\tilde{X},E}^{p,\bullet}(\tilde{W})) \ar[r]^{} &
   \Gamma(\tilde{X},\A_{\tilde{X}}^{p,\bullet}(\tilde{W})) \ar[r]^{\jmath^{\ast}} &
   \Gamma(E,\A_{E}^{p,\bullet}(\jmath^{\ast}\tilde{W})) \ar[r] &0}
\end{equation}
 of short exact sequences for complexes of vector spaces.

By the definition of hypercohomologies,
the standard diagram chasing of the commutative diagram \eqref{commut-complexes} gives rise to a ladder of hypercohomologies
\begin{equation}\label{long-exact-cohom1}
\small{
\xymatrix@C=0.3cm{
   \cdots \ar[r]^{} & \mathbb{H}^{q}(X,\K_{X,Z}^{p,\bullet}(W)) \ar[d]_{\pi^{\ast}} \ar[r]^{} &\mathbb{H}^{q}(X,\A_{X}^{p,\bullet}(W)) \ar[d]_{\pi^{\ast}} \ar[r]^{} & \mathbb{H}^{q}(Z, \A_{Z}^{p,\bullet}(\imath^{\ast}W))\ar[d]_{\varpi^{\ast}} \ar[r]^{} & \mathbb{H}^{q+1}(X,\K_{X,Z}^{p,\bullet}(W))\ar[d]_{\pi^{\ast}} \ar[r]^{} & \cdots \\
   \cdots \ar[r] & \mathbb{H}^{q}(\tilde{X},\K_{\tilde{X},E}^{p,\bullet}(\tilde{W})) \ar[r]^{} &
  \mathbb{H}^{q}(\tilde{X},\A_{\tilde{X}}^{p,\bullet}(\tilde{W})) \ar[r]^{} &
   \mathbb{H}^{q}(E,\A_{E}^{p,\bullet}(\jmath^{\ast}\tilde{W}))\ar[r]^{} &
   \mathbb{H}^{q+1}(\tilde{X},\K_{\tilde{X},E}^{p,\bullet}(\tilde{W})) \ar[r] & \cdots.}
   }
\end{equation}

By Proposition \ref{redol} (i),
$\K_{X,Z}^{p,\bullet}$ is  a fine resolution of $\K_{X,Z}^{p}$ and thus $\K_{X,Z}^{p,\bullet}(W)$ is a fine resolution of $\K_{X,Z}^{p}(W)$.
It follows
$\mathbb{H}^{q}(X,\K_{X,Z}^{p,\bullet}(W))\cong H^{q}(X,\K_{X,Z}^{p}(W))$.
Similarly, we have $$\mathbb{H}^{q}(\tilde{X},\K_{\tilde{X},E}^{p,\bullet}(\tilde{W}))\cong H^{q}(\tilde{X},\K_{\tilde{X},E}^{p}(\tilde{W})).
$$
Moreover, by the bundle-valued Dolbeault theorem to $(X, W)$, $(Z, \imath^{\ast}W)$, $(\tilde{X}, \tilde{W})$ and $(E, \jmath^{\ast}\tilde{W})$,
the ladder \eqref{long-exact-cohom1} turns into
\begin{equation}\label{comm-cohom-long-seq}
\small{
\xymatrix@C=0.3cm{
  \cdots \ar[r]^{}
  & H^{q}(X,\mathscr{K}^{p}_{X,Z}(W))\ar[d]_{\pi^{\ast}} \ar[r]^{}
  & H^{q}(X,\Omega^{p}_{X}(W))\ar[d]_{\pi^{\ast}} \ar[r]^{}
  & H^{q}(Z,\Omega^{p}_{Z}(\imath^{\ast}W)) \ar[d]_{\pi^{\ast}} \ar[r]^{}
  & H^{q+1}(X,\mathscr{K}^{p}_{X,Z}(W))\ar[d]_{\pi^{\ast}} \ar[r]^{} & \cdots \\
   \cdots \ar[r]
  & H^{q}(\tilde{X}, \mathscr{K}^{p}_{\tilde{X},E}(\tilde{W}))\ar[r]^{}
  & H^{q}(\tilde{X},\Omega^{p}_{\tilde{X}}(\tilde{W})) \ar[r]^{}
  & H^{q}(E, \Omega^{p}_{\tilde{X}}(\jmath^{\ast}\tilde{W})) \ar[r]
  & H^{q+1}(\tilde{X}, \mathscr{K}^{p}_{\tilde{X},E}(\tilde{W}))\ar[r]&\cdots.}
}
\end{equation}
Combine \eqref{comm-cohom-long-seq} with Lemma \ref{kersheaf-iso} to get the equivalent ladder
\begin{equation}\label{comm-cohom-long-seq-1}
\small{
\xymatrix@C=0.3cm{
  \cdots \ar[r]^{}
  & H^{q}(X,\mathscr{K}^{p}_{X,Z}(W))\ar[d]_{\cong} \ar[r]^{}
  & H^{q}(X,\Omega^{p}_{X}(W))\ar[d]_{\pi^{\ast}} \ar[r]^{}
  & H^{q}(Z,\Omega^{p}_{Z}(\imath^{\ast}W)) \ar[d]_{\varpi^{\ast}} \ar[r]^{}
  & H^{q+1}(X,\mathscr{K}^{p}_{X,Z}(W))
  \ar[d]_{\cong} \ar[r]^{} & \cdots \\
   \cdots \ar[r]
  & H^{q}(\tilde{X}, \mathscr{K}^{p}_{\tilde{X},E}(\tilde{W}))\ar[r]^{}
  & H^{q}(\tilde{X},\Omega^{p}_{\tilde{X}}(\tilde{W})) \ar[r]^{}
  & H^{q}(E,\Omega^{p}_{E}(\jmath^{\ast}\tilde{W})) \ar[r]
  & H^{q+1}(\tilde{X}, \mathscr{K}^{p}_{\tilde{X},E}(\tilde{W}))\ar[r]&\cdots}
}
\end{equation}
to \eqref{comm-cohom-long-seq}.
Since $\pi: \tilde{X}\rightarrow X$ is a proper surjective holomorphic map,
by \cite[Theorem 3.1]{Wel74}, the pullback morphism
$$
\pi^{\ast}: H^{q}(X,\Omega^{p}_{X}(W)) \rightarrow H^{q}(\tilde{X},\Omega^{p}_{\tilde{X}}(\tilde{W}))
$$
is injective.
In particular, the morphism $\varpi^{\ast}$ in \eqref{comm-cohom-long-seq-1} is injective since the Weak Five Lemma (cf. \cite[Lemma 3.3 (i) of Chapter I]{Mac63}).
Due to Proposition \ref{coker=coker}, the commutative diagram \eqref{comm-cohom-long-seq-1} gives rise to the following isomorphism of complex vector spaces
\begin{equation}\label{coker-coker}
\mathrm{coker}\,\bigl(H^{q}(X,\Omega^{p}_{X}(W))
\stackrel{\pi^{\ast}}\rightarrow
H^{q}(\tilde{X},\Omega^{p}_{\tilde{X}}(\tilde{W}))\bigr)
\cong
\mathrm{coker}\,\bigl(
H^{q}(Z,\Omega^{p}_{Z}(\imath^{\ast}W))
\stackrel{\varpi^{\ast}}\rightarrow
H^{q}(E,\Omega^{p}_{E}(\jmath^{\ast}\tilde{W}))\bigr).
\end{equation}

From definition, the exceptional divisor $E$ is biholomorphic to the projectivization of the normal bundle $\mathcal{N}_{Z/X}$.
Owing to the Dolbeault theorem, we get that the morphism
$$
\varpi^{\ast}:
H^{q}(Z,\Omega^{p}_{Z}(\imath^{\ast}W))
\rightarrow
H^{q}(E,\Omega^{p}_{E}(\jmath^{\ast}\tilde{W}))
$$
is equal to the morphism
$$
\varpi^{\ast}:
H^{p,q}(Z,\imath^{\ast}W)
\rightarrow
H^{p,q}(E,\jmath^{\ast}\tilde{W})
$$
in Lemma \ref{dolb-projective-formula}.
Thanks to Lemma \ref{dolb-projective-formula} we get

\begin{eqnarray}\label{coker-pi-E}
  \frac{H^{p,q}(E,\jmath^{\ast}\tilde{W})}
{\varpi^{\ast}H^{p,q}(Z,\imath^{\ast}W)}
&=&
\bigoplus_{i=1}^{r-1}\tilde{\bm{t}}^{i}
\wedge\varpi^{\ast}H^{p-i,q-i}(Z,\imath^{\ast}W)
\nonumber \\
&\cong&
\bigoplus_{i=1}^{r-1}H^{p-i,q-i}(Z,\imath^{\ast}W).
\end{eqnarray}
By \eqref{coker-coker} and \eqref{coker-pi-E} one has the following isomorphism of complex vector spaces
\begin{eqnarray}\label{thm1-iso-0}
H^{p,q}(\tilde{X},\tilde{W})
&\cong&
H^{q}(\tilde{X},\Omega^{p}_{\tilde{X}}(\tilde{W}))
\nonumber\\
&\cong&
H^{q}(X,\Omega^{p}_{X}(W))
\oplus
\mathrm{coker}\,\varpi^{\ast} \nonumber\\
&\cong &
H^{p,q}(X,W)
\oplus \Big(\bigoplus_{i=1}^{r-1}
H^{p-i,q-i}(Z,\imath^{\ast}W)\Big).
\end{eqnarray}

Next we verify the injectivity of $\phi$.
From \eqref{coker-coker}, we obtain a commutative diagram of short exact sequences as follows
\begin{equation}\label{comm-coker}
\xymatrix@C=0.5cm{
0 \ar[r]^{} & H^{p,q}(X,W)
\ar[d]_{\imath^{\ast}}
\ar[r]^{\pi^{\ast}} &
H^{p,q}(\tilde{X},\pi^{\ast}W)
\ar[d]_{\jmath^{\ast}}
\ar[r]^{} &
\mathrm{coker}\,(\pi^{\ast}) \ar[d]_{\bar{\jmath}^{\ast}}^{\cong} \ar[r]^{} & 0 \\
0 \ar[r] &
H^{p,q}(Z,\imath^{\ast}W)
\ar[r]^{\varpi^{\ast}} &
H^{p,q}(E,\jmath^{\ast}\tilde{W})
\ar[r]^{} &
\mathrm{coker}\,(\varpi^{\ast})\ar[r]^{} & 0.}
\end{equation}
Here $\bar{\jmath}^{\ast}$ is the induced isomorphism of the quotient spaces by $\jmath^{\ast}$.
Combining \eqref{form-curr} with \eqref{comm-coker} we get the following diagram
\begin{equation}\label{curr-coker}
\xymatrix@C=0.5cm{
& H^{p,q}_{\mathscr{D}}(X,W)
 &
H^{p,q}_{\mathscr{D}}(\tilde{X},\pi^{\ast}W)
\ar[l]_{\pi_{\flat}}
&
& \\
0 \ar[r]^{} & H^{p,q}(X,W)
\ar[u]^{\varrho_{\ast}}_{\cong}
\ar[d]_{\imath^{\ast}}
\ar[r]^{\pi^{\ast}} &
H^{p,q}(\tilde{X},\pi^{\ast}W)
\ar[d]_{\jmath^{\ast}}
\ar[u]^{\tilde{\varrho}_{\ast}}_{\cong}
\ar[r]^{} &
\mathrm{coker}\,(\pi^{\ast}) \ar[d]_{\bar{\jmath}^{\ast}}^{\cong} \ar[r]^{} & 0 \\
0 \ar[r] &
H^{p,q}(Z,\imath^{\ast}W)
\ar[r]^{\varpi^{\ast}} &
H^{p,q}(E,\jmath^{\ast}\tilde{W})
\ar[r]^{} &
\mathrm{coker}\,(\varpi^{\ast})\ar[r]^{} & 0.}
\end{equation}
Recall that
$
\pi_{\ast}=
\varrho^{-1}_{\ast}\circ\pi_{\flat}
\circ\tilde{\varrho}_{\ast}
$
and
$
\phi=\pi_{\ast}+\sum^{r-1}_{i=1}\Pi_{i}\circ\jmath^{\ast}
$
maps $H^{p,q}(\tilde{X},\pi^{\ast}W)$ to
$$
H^{p,q}(X,W)\oplus
\biggl(\bigoplus^{r-1}_{i=1}
H^{p-i,q-i}(Z,\imath^{\ast}W)\biggr).
$$

Notice that $\pi_{\ast}$ is surjective;
moreover, a direct check shows that
the kernel of $\pi_{\ast}$ is equal to the co-kernel of $\pi^{\ast}$, i.e.,
$$
\ker\,(\pi_{\ast})=\mathrm{coker}\,(\pi^{\ast})
\stackrel{\bar{\jmath}^{\ast}}\rightarrow
\biggl(\bigoplus^{r-1}_{i=1}
\bm{h}^{i}\wedge\varpi^{\ast}
H^{p-i,q-i}(Z,\imath^{\ast}W)\biggr).
$$
It follows that the restriction of $\jmath^{\ast}$ on
$\ker\,(\pi_{\ast})$ is injective.
Given an element
$\tilde{\alpha}\in H^{p,q}(\tilde{X},\pi^{\ast}W)$, suppose that $\phi(\tilde{\alpha})=0$.
Then we get
$\tilde{\alpha}\in\ker\,(\pi_{\ast})$,
$\jmath^{\ast}(\tilde{\alpha})=0$ and hence $\tilde{\alpha}=0$.
This implies that $\phi$ is injective
and the proof of the theorem is now complete.

\section{Applications of Theorem \ref{thm1}}\label{sec-6}

We start this section with several basic notions in bimeromorphic geometry, whose nice reference is \cite[$\S$ 2]{Uen75}.
The first one is the proper modification.
\begin{defn}
A morphism $\pi: \tilde{X}\rightarrow X$ of two equidimensional complex spaces is called a \emph{proper modification}, if it satisfies:
\begin{enumerate}
  \item [(i)] $\pi$ is proper and surjective;
  \item [(ii)] there exist nowhere dense analytic subsets $\tilde E\subset \tilde{X}$ and $E \subset X$ such that
                  $$
                  \pi:\tilde{X}-\tilde E\rightarrow X- E
                  $$
                  is a biholomorphism, where $\tilde E:=\pi^{-1}(E)$ is called the \emph{exceptional space of the modification}.
\end{enumerate}
If $\tilde X$ and $X$ are compact, a proper modification $\pi: \tilde{X}\rightarrow X$ is often called simply a \emph{modification}.
\end{defn}

More generally, we have the following important notation in complex geometry.
\begin{defn}\label{bimero}
Let $X$ and $Y$ be two complex spaces.
A map $\varphi$ of $X$ into the power set of $Y$ is a \emph{meromorphic map} of $X$ into $Y$,
denoted by $\varphi: X\dashrightarrow Y$, if $X$ satisfies the following conditions:
\begin{enumerate}
  \item [(i)] The graph $G_{\varphi}=\{(x,y)\in X\times Y\ |\ y\in \varphi(x)\}$ of $\varphi$ is an irreducible analytic subset in $X\times Y$;
  \item [(ii)] The projection map $P_X:G_{\varphi}\rightarrow X$ is a proper modification.
\end{enumerate}

A meromorphic map $\varphi: X\dashrightarrow Y$ of complex varieties is called a \emph{bimeromorphic map} if $P_Y:G_{\varphi}\rightarrow Y$ is also a proper modification.

If $\varphi$ is a bimeromorphic map, the analytic set
$$
\{(y,x)\in Y\times X\ |\ (x,y)\in G_\varphi\}\subset Y\times X
$$
defines a meromorphic map $\varphi^{-1}:Y\dashrightarrow X$ such that $\varphi\circ\varphi^{-1}=\mathrm{id}_Y$ and $\varphi^{-1}\circ\varphi=\mathrm{id}_X$.

Two complex varieties $X$ and $Y$ are called \emph{bimeromorphically equivalent} (or \emph{bimeromorphic}) if there exists a bimeromorphic map $\varphi: X\dashrightarrow Y$.
\end{defn}

Relating blow-up to bimeromorphic map, one is fortunate to own the remarkable:
\begin{thm}[{\cite[Theorem 0.3.1]{AKMW02} and \cite{Wlo03}}]\label{wft}
Let $\pi:\tilde{X}\dashrightarrow X$ be a bimeromorphic map between two compact complex manifolds  as in Definition \ref{bimero}.
Let $U$ be an open set where $\pi$ is an isomorphism.
Then $\pi$ can be factored into a sequence of blow-ups and blow-downs along irreducible nonsingular centers disjoint from $U$.
That is, to any such $\pi$ we associate a diagram
\begin{equation}\label{wf}
\pi:\tilde{X}=X_0\stackrel{\pi_1}{\dashrightarrow} X_1\stackrel{\pi_2}{\dashrightarrow}\cdots\stackrel{\pi_{i-1}}{\dashrightarrow} X_{i-1}\stackrel{\pi_{i}}{\dashrightarrow} X_{i}\stackrel{\pi_{i+1}}{\dashrightarrow}\cdots\stackrel{\pi_{l}}{\dashrightarrow} X_l=X,
\end{equation}
where
\begin{enumerate}
    \item[(i)]
              $\pi=\pi_l\circ\cdots\circ\pi_1;$
    \item [(ii)]
              $\pi_i$ are isomorphisms on $U$;
    \item [(iii)]
either $\pi_i:X_{i-1}\dashrightarrow X_{i}$ or $\pi_i^{-1}:X_{i}\dashrightarrow X_{i-1}$ is a morphism obtained by blowing up a nonsingular center disjoint from $U$.
 \end{enumerate}
\end{thm}

Now we state several applications of Theorem \ref{thm1} to complex algebraic geometry.

\subsection{Vanishing theorems}

 Girbau's
theorem deals with the vanishing theorem on higher-order cohomology of the line bundle being (possibly everywhere) degenerate semi-positive.
\begin{thm} Let $L$ be a holomorphic line bundle over a compact connected $n$-dimensional K\"ahler manifold $X$.
\begin{enumerate}
    \item[$(i)$] \label{Girbau0}
If the Chern curvature form $i\Theta(L)$ of $L$ is semi-positive and has at least $n-s+1$ positive eigenvalues
at a point of $X$ for some integer $s\in \{1,\cdots, n\}$, then
$H^q(X,K_X\otimes L) = 0$ for $q \geq s$. Here $K_X$ is canonical bundle of $X$.

    \item[$(ii)$]{\emph{\textbf{(\cite{Gir76})}} \label{Girbau1}
If $i\Theta(L)$ is semi-positive and has at least $n-s+1$ positive eigenvalues
at every point of $X$ for some integer $s\in \{1,\cdots, n\}$, then
$H^q(X,\Omega^p\otimes L) = 0$ for $p+q \geq n+s$.}
 \end{enumerate}
\end{thm}

From  Theorem \ref{thm1} we can construct a series of examples for the above theorem avoiding the complicated computation.
\begin{ex}[{Ramanujam's example, \cite{Ram}, \cite[(4.5) Remark of Chapter VII]{Dem12}}] \label{Ramexample}
Here is a quick description of Ramanujam's example, which shows that Girbau's result is no longer true for $p<n$ when $i\Theta(L)$ is semi-positive and has at least $n-s+1$ positive eigenvalues even on a dense open set: Let
$$\pi: X\rightarrow \mathbb{C}\mathrm{P}^n,\quad n\geq 3$$
be the blow-up of $\mathbb{C}\mathrm{P}^n$ at a point $\{a\}$.
Then for any $m>0$, $\pi^*\mathcal{O}_{\mathbb{C}\mathrm{P}^n}(m)$ is a semipositive line bundle over $X$, positive $(s=1)$ outside $\pi^{-1}(\{a\})$ and for $1\leq p\leq n-1$,
\begin{eqnarray*}
H^p(X,\Omega_X^p\otimes \pi^*\mathcal{O}_{\mathbb{C}\mathrm{P}^n}(m))
&\cong&
H^p(\mathbb{C}\mathrm{P}^n,\Omega_{\mathbb{C}\mathrm{P}^n}^p\otimes \mathcal{O}_{\mathbb{C}\mathrm{P}^n}(m))\oplus \biggl(\bigoplus^{n-1}_{i=1}H^{p-i,p-i}(\{a\},\mathcal{O}_{\mathbb{C}\mathrm{P}^n}(m)|_{\{a\}})\biggr)\\
&\cong&
\bigoplus^{n-1}_{i=1}H^{p-i,p-i}(\{a\},\mathcal{O}_{\mathbb{C}\mathrm{P}^n}(m)|_{\{a\}}) \\
&\neq& 0
\end{eqnarray*}
directly by Bott vanishing theorem (\cite[Theorem 5.2]{Kod05} for example) and Theorem \ref{thm1}.
\end{ex}
\begin{ex}[] Here is a new positive example, which affirms that Girbau's result is possibly true for $p\neq q$ when $i\Theta(L)$ is semi-positive and positive only on a dense open set: Let $\pi: X\rightarrow \mathbb{C}\mathrm{P}^n$ $(n\geq 3)$ be the blowing up at a point $\{a\}$ in $\mathbb{C}\mathrm{P}^n$. Then for $p\neq q$,
$$H^q(X,\Omega_X^p\otimes \pi^*\mathcal{O}_{\mathbb{C}\mathrm{P}^n}(1))\cong H^q(\mathbb{C}\mathrm{P}^n,\Omega_{\mathbb{C}\mathrm{P}^n}^p\otimes \mathcal{O}_{\mathbb{C}\mathrm{P}^n}(1))\oplus
\biggl(\bigoplus^{n-1}_{i=1}H^{p-i,q-i}(\{a\},\mathcal{O}_{\mathbb{C}\mathrm{P}^n}(1)|_{\{a\}})\biggr)=0$$
directly by Theorem \ref{thm1} and Bott vanishing theorem. It seems interesting to give a direct proof of the above vanishing.
\end{ex}

Actually, Ramanujam \cite{Ram} generalized the Kodaira's vanishing theorem for big and nef line bundles on smooth projective surfaces.
Later, a higher dimensional analogy of Kodaira-Ramanujam vanishing was independently proved by Kawamata \cite{Kaw82} and Viehweg \cite{Vie82}.
By contrast of Kodaira-Nakano vanishing theorem,
the Nakano-type Kawamata-Viehweg vanishing can fail for big and nef line bundles; see for instance \cite[Example 4.3.4]{Laz04}.
Here we present a slight generalization of \cite[Example 4.3.4]{Laz04} since $\pi^{\ast}\mathcal{O}_{\mathbb{C}\mathrm{P}^{n}}(1)$ is a big and nef line bundle.
\begin{ex}[]\label{Ramexample-genl} Let $\pi: X\rightarrow \mathbb{C}\mathrm{P}^n$ $(n\geq 3)$ be the blow-up along a smooth curve $C$.
Then $\pi^*\mathcal{O}_{\mathbb{C}\mathrm{P}^n}(1)$ is a semipositive line bundle over $X$, positive $(s=1)$ outside $\pi^{-1}(C)$ and a direct application of Theorem \ref{thm1} implies that for $1\leq p\leq n-2$,
\begin{eqnarray*}
H^{p}(X,\Omega_X^p\otimes \pi^*\mathcal{O}_{\mathbb{C}\mathrm{P}^n}(1))
&\cong&
H^{p}(\mathbb{C}\mathrm{P}^n,\Omega_{\mathbb{C}\mathrm{P}^n}^p\otimes \mathcal{O}_{\mathbb{C}\mathrm{P}^n}(1))\oplus
\biggl(\bigoplus^{n-2}_{i=1}H^{p-i,p-i}(C,\mathcal{O}_{\mathbb{C}\mathrm{P}^n}(1)|_{C})\biggr)\\
&\cong& \bigoplus^{n-2}_{i=1}H^{p-i,p-i}(C,\mathcal{O}_{\mathbb{C}\mathrm{P}^n}(1)|_{C}) \\
&\cong& H^{0}(C,\mathcal{O}_{\mathbb{C}\mathrm{P}^n}(1)|_{C})\oplus\cdots \\
&\neq& 0.
\end{eqnarray*}
\end{ex}

Now let us present a counterexample to Nakano-type generic vanishing theorem.
Let $X$ be a compact, connected K\"{a}hler manifold. We denote by $\mathrm{Pic}^{0}(X)$ the identity component of the Picard group of $X$ and $\mathfrak{A}:X\rightarrow \mathrm{Alb}(X)$ the Albanese mapping of $X$.
In \cite[Theorem 1]{GL87}, Green-Lazarsfeld obtained the Kodaira-type {\it generic vanishing theorem}:
if $L\in \mathrm{Pic}^{0}(X)$ is a generic line bundle, then $H^{q}(X,L)=0$ for $q<\mathrm{dim}\, \mathfrak{A}(X)$.
However, for $p+q<\mathrm{dim}\, \mathfrak{A}(X)$,
$$H^{q}(X,\Omega_{X}^{p}\otimes L)\neq0,$$
which means that the Nakano-type generic vanishing theorem does not hold generally.

\begin{ex}[{\cite[\S 3, Remark]{GL87}}]
Suppose that $X$ is an abelian variety of dimension $n\geq 4$ and $Z\subset X$ is a smooth curve of genus $g>1$.
Let $\mathfrak{A}:\tilde{X}\rightarrow X$ be the blow-up of $X$ centered at $Z$.
Then $\mathfrak{A}$ is the Albanese mapping of $\tilde{X}$ and $\mathrm{dim}\, \mathfrak{A}(\tilde{X})=n$.
Using Theorem \ref{thm1}, we have
$$
H^{q}(\tilde{X}, \Omega_{\tilde{X}}^{p}\otimes \mathfrak{A}^{\ast}L)
\cong
H^{q}(X, \Omega_{X}^{p}\otimes L)\oplus \bigoplus_{i=1}^{n-2}H^{q-i}(Z, \Omega_{Z}^{p-i}\otimes L|_{Z})
$$
for any non-zero generic $L\in \mathrm{Pic}^{0}(X)$ and $\mathfrak{A}^{\ast}L\in \mathrm{Pic}^{0}(\tilde{X})$.
In particular, for $1\leq p<\frac{n-1}{2}$, the generic vanishing theorem implies that
\begin{eqnarray*}
H^{p+1}(\tilde{X}, \Omega_{\tilde{X}}^{p}\otimes \mathfrak{A}^{\ast}L)
&\cong&
H^{p+1}(X, \Omega_{X}^{p}\otimes L)\oplus \bigoplus_{i=1}^{n-2}H^{p+1-i}(Z, \Omega_{Z}^{p-i}\otimes L|_{Z}) \\
&\cong&
\bigoplus_{i=1}^{n-2}H^{p+1-i}(Z, \Omega_{Z}^{p-i}\otimes L|_{Z}) \\
&=& H^{1}(Z, \mathcal{O}_{Z} \otimes L|_{Z})\oplus\cdots\\
&\neq& 0,
\end{eqnarray*}
since $\Omega_{X}^{p}\cong \mathcal{O}_{X}^{\oplus \binom{n}{p}}$ and $c_{1}(\mathcal{O}_{Z} \otimes L|_{Z})=0$,
 and by Riemann-Roch theorem,
$$h^{0}(Z, \mathcal{O}_{Z} \otimes L|_{Z})-h^{1}(Z, \mathcal{O}_{Z} \otimes L|_{Z})=1-g<0$$ for $g>1$.
\end{ex}

\subsection{Blow-up invariants}
In this subsection, we always consider the blow-up $\pi:\tilde{X}\rightarrow X$ of a compact complex $n$-dimensional manifold $X$ with a center $\imath: Z\subset X$ and are able to obtain all the results applicable to a composition of finite blow-ups.

 The first consideration is the blow-up invariance for the volume of a line bundle.
Recall that
\begin{defn}[Volume of a line bundle] Let $X$ be an irreducible
compact complex space of dimension $n$, and $L$ a line bundle on $X$. The
\emph{volume} of $L$ is defined to be the non-negative real number
$$vol(L) = vol_X(L) = \limsup_{m\rightarrow \infty}\frac{h^0(X,L^{\otimes m})}{m^n/n!}.$$
The volume $vol(D)=vol_X(D)$ of a Cartier divisor $D$ is defined similarly, or
by passing to $\mathcal{O}_X(D)$.
\end{defn}

From  Fujita's vanishing theorem and Properties of the volume in \cite[Proposition 2.2.35]{Laz04}, it follows that the volume of a line bundle depends only upon its numerical equivalence class.
\begin{prop}[Blow-up invariance of volume]\label{biv}
 Given an integral or $\mathbb{Q}$-divisor $D$ on
X, put $\tilde D=\pi^*D.$ Then
$$vol_{\tilde X}(\tilde D)= vol_X(D).$$
\end{prop}
\begin{proof}
This is a direct corollary of Theorem \ref{thm1} when one takes $p=q=0$ and thus for each $i\geq 1$, $$H^{p-i,q-i}(Z;\imath^{*}\mathcal{O}_X(mD))=0.$$
Comparing birational invariance I of volume as in  {\cite[Proposition 2.2.43]{Laz04}}, one can obtain the estimate:
$$h^0(X,\mathcal{O}_X(mD))\leq h^0(\tilde X,\mathcal{O}_{\tilde X}(m\tilde D))=h^0(X,\mathcal{O}_X(mD))+O(m^{n-1})$$
when $\pi$ is a birational projective mapping.
\end{proof}

Next, we consider the bimeromorphic invariance of plurigenera and Kodaira dimension, and blow-up invariance of Kodaira-Iitaka dimension of holomorphic line bundle on a compact complex manifold.
We first give a geometric description of the Kodaira map. Let $X$ be a compact connected complex manifold of dimension $n$ and $L$ a holomorphic line bundle over $X$. The space for  holomorphic sections of $L$ on $X$ is finite-dimensional. Set the linear system associated to $V=H^{0}(X,L)$ as $|V|=\{\Div(s): s\in V\}$. The \emph{base point locus} of the linear system $|V|$ is given by $$\Bl_{|V|}=\cap_{s\in V} \Div(s)=\{x\in X: s(x)=0,\ \text{for all $s\in V$}\}.$$
Set $d=\dim H^{0}(X,L)$ and let $\mathbb{G}(d-1,H^{0}(X,L))$ be the Grassmannian of hyperplanes of $H^{0}(X,L)$.
\begin{defn}
  The \emph{Kodaira map} $\Phi_V$ associated to $L$ is defined by
  $$\Phi_V: X\setminus \Bl_{|V|}\longrightarrow \mathbb{G}(d-1,H^{0}(X,L)): x\longmapsto \{s\in H^{0}(X,L): s(x)=0\}.$$
\end{defn}

Now consider $V_m:=H^{0}(X,L^{\otimes m})$ and $\Phi_m:=\Phi_{V_m}$.
Set
$$
\varrho_m=
    \begin{cases}
      \max\{\rk\ \Phi_m: x\in X\setminus \Bl_{|V_m|}\},\quad &\text{if $V_m\neq\{0\}$,}\\
      -\infty,\quad &\text{otherwise.}
    \end{cases}
$$
The \emph{Kodaira-Iitaka dimension} of $L$ is
$$\kappa(L)=\max\{\varrho_m:m\in \mathbb{N}^+\}.$$
The \emph{Kodaira dimension} $\kappa(X)$ of a compact complex manifold $X$ is defined as $\kappa(K_X)$ of the canonical bundle $K_X$.
A useful characterization of Kodaira-Iitaka dimension is:
\begin{thm}[{\cite[Theorem $8.1$]{Uen75}}]\label{asymp}
For a Cartier divisor (or a line bundle) $D$ on a variety $V$, there exists positive numbers $\alpha,\beta$ and a positive integer $m_0$ such that for any integer $m\geq m_0$, there hold the inequalities
\begin{equation}\label{asymp-ineq}
 \alpha m^{\kappa(D)}\leq P_m(dD):=\dim_{\mathbb{C}}H^0 (V, \mathcal{O}_V(mdD))\leq \beta m^{\kappa(D)},
\end{equation}
where $d$ is some positive integer depending on $D$. When the divisor $D$ is effective (or $P_1(D)\neq 0$), one can take $d=1$ in \eqref{asymp-ineq}.
\end{thm}

\begin{cor}\label{plg} Let $\nu: \tilde{X}\rightarrow X$ be a
bimeromorphic map of two compact complex manifolds $\tilde{X}$ and $X$.
Then there hold the equalities
$$P_m(\tilde{X})=P_m({X}),\quad \kappa(\tilde{X})=\kappa({X})$$ for $P_m(\bullet)=\dim_{\mathbb{C}}H^0({\bullet}, K_{{\bullet}}^{\otimes m})$. In particular, if $\pi:\tilde{X}\rightarrow X$ is the blow-up morphism and $L$ is a  holomorphic line bundle over $X$, then
$$\kappa(L)=\kappa(\pi^*L).$$
\end{cor}
\begin{proof} Compare the classical proof in \cite[Lemma 6.3, Corollary 6.4, Theorem 5.13]{Uen75} or  \cite[Example 2.1.16]{Laz04}.

Without loss of generality, one may assume that $\nu$ is a blow-up morphism along the smooth center with codimension $r$ with the exceptional divisor $E$ by the weak factorization Theorem \ref{wft}. Then
$$\nu^*: H^0({X}, K_{{X}}^{\otimes m}) \rightarrow H^0(\tilde{X}, K_{\tilde{X}}^{\otimes m})$$
is an isomorphism. Actually, notice
$$K_{\tilde{X}}=\nu^* K_X \otimes \mathcal{O}((r-1)E)$$ as shown in
\cite[(12.7) Proposition of Chapter VII]{Dem12}, and thus
$$H^0(\tilde{X}, \nu^*(K_{{X}}^{\otimes m}))\stackrel{\cong}{\longrightarrow} H^0(\tilde{X}, \nu^*(K_{{X}}^{\otimes m})\otimes \mathcal{O}(m(r-1) E))=H^0(\tilde{X}, K_{\tilde{X}}^{\otimes m}),$$
where the group on the left-hand side is isomorphic to $H^0({X}, K_{{X}}^{\otimes m})$ directly by Theorem \ref{thm1}, while a section of the group $H^0(\tilde{X}, \nu^*(K_{{X}}^{\otimes m})\otimes \mathcal{O}(m(r-1) E))$
 restricts to one of $K_{{X}}^m$ defined off the codimension $\geq 2$ fundamental locus of $\nu$, which extends to the whole $X$ by normality.

 Then the Kodaira(-Iitaka) dimension case is a direct result of Theorem \ref{asymp} (and Theorem \ref{thm1}).
\end{proof}

Similarly, one can also obtain that the \emph{holomorphic Euler characteristic }
$$\chi(X,F)=\sum_{i=1}^{\dim_\mathbb{C} X}(-1)^{i}\dim_\mathbb{C} H^i(X,F)$$
for a holomorphic vector bundle $F$ over a compact complex manifold $X$ is a blow-up invariant in the sense:
\begin{cor}\label{hec} With the above setting,
there holds the equality
$$\chi(X,F)=\chi(\tilde X, \pi^*F).$$
\end{cor}

Notice that it seems not easy to obtain Corollary \ref{hec} by the remarkable Hirzebruch-Riemann-Roch theorem, which is usually used to compute the holomorphic Euler characteristic, that is, for a holomorphic vector bundle $F$ over $X$, the holomorphic Euler characteristic $\chi(X,F)$ can be computed as
  $$\chi(X,F)=\int_X ch(F)\cdot td(X),$$
  where the exponential Chern character $ch(F)$ is $$\rk\ F+c_1(F)+\frac{1}{2!}(c_1^2(F)-c_2(F))+\frac{1}{3!}(c_1^3(F)-3c_1(F)c_2(F)+3c_3(F))+\cdots,$$
  and
  the Todd class $td(X)$ of the holomorphic tangent bundle $T^{\prime}_X$ of $X$ is
  $$1+\frac{1}{2}c_1(T^{\prime}_X)+\frac{1}{12}(c_1^2(T^{\prime}_X)+c_2(T^{\prime}_X))+\frac{1}{24}c_1(T^{\prime}_X)c_2(T^{\prime}_X)+\cdots.$$

It is worth introducing:
\begin{defn}
Let $X$ and $Y$ be compact complex manifolds.
We say that $X$ and $Y$ are {\it strongly K-equivalent} if,
there are the correspondences
$$
\xymatrix{
& Z \ar[dl]_{\pi_X} \ar[dr]^{\pi_Y} \\
X   &&     Y }
$$
where $\pi_X, \pi_Y$ are the compositions of finite blow-ups, and $\pi_X^{\ast} K_{X}\cong \pi_Y^{\ast}K_{Y}$.
\end{defn}
Then with this setting, the similar argument to Corollary \ref{plg} gives
$$
\chi(X, K_{X}^{\otimes m})=\chi(Y, K_{Y}^{\otimes m}).
$$

\subsection{Blow-up formula for bundle-valued Hochschild homology}

Let $X$ be a compact complex manifold of dimension $n$ and $W$ a locally sheaf of constant rank on $X$.
Recall that the {\it Hochschild homology of $X$ with value in $W$} is given by
$$
\mathrm{HH}_{k}(X, W):=
\mathrm{Tor}_{k}^{X\times X}(\mathcal{O}_{\Delta}, \Delta_{\ast}W), \,\;-n\leq k\leq n,
$$
where $\mathcal{O}_{\Delta}$ is the structure sheaf
of the diagonal embedding $\Delta: X \hookrightarrow X\times X$.
By the Hochschild-Kostant-Rosenberg Theorem as in \cite[Corollary 3.1.4]{BF08},
we have
$$
\mathrm{HH}_{k}(X, W)
\cong
\bigoplus_{p-q=k} H^{q}(X, \Omega_{X}^p\otimes W)
\cong
\bigoplus_{p-q=k} H^{p, q}(X, W)
$$
for any $-n\leq k\leq n$.

\begin{cor}
Let $Z\subset X$ be a closed complex submanifold of complex codimension $r\geq2$.
For any $-n\leq k\leq n$,
there is an isomorphism of Hochschild homologies
$$\label{buf}
\mathrm{HH}_{k}(\tilde{X}, \tilde{W})
\cong
\mathrm{HH}_{k}(X, W) \oplus \mathrm{HH}_{k}(Z, \imath^*W)^{\oplus (r-1)},
$$
where $\pi:\tilde{X}\rightarrow X$ is the blow-up of $X$ along $Z$, and $\tilde{W}:=\pi^{\ast}W$.
\end{cor}

\begin{proof}
This is a combination result of our main Theorem \ref{thm1}
and HKR theorem and see also \cite[Corollary 1.7]{RYY}$_{v3}$.
In fact,  we have the following isomorphisms:
\begin{eqnarray*}
\mathrm{HH}_{k}(\tilde{X}, \tilde{W})
&\cong& \bigoplus_{p-q=k} H^{p, q}(\tilde{X}, \tilde{W}) \;\;(\textrm{HKR Theorem for $(\tilde{X},\tilde{W})$})\\
&\cong &\bigoplus_{p-q=k}  \Big( H^{p,q}(X, W)\oplus \bigoplus_{i=1}^{r-1} H^{p-i,q-i}(Z, \imath^*W) \Big) \;\; (\textrm{Theorem \ref{thm1}}) \\
&\cong & \bigoplus_{p-q=k} H^{p, q}(X,W) \oplus
\bigoplus_{i=1}^{r-1} \Big( \bigoplus_{p-q=k} H^{p-i, q-i}(Z, \imath^*W) \Big)\\
&\cong & \mathrm{HH}_{k}(X, W) \oplus \mathrm{HH}_{k}(Z, \imath^*W)^{\oplus (r-1)} \;\;(\textrm{HKR theorem for $(X, W)$, $(Z, \imath^*W)$})
\end{eqnarray*}
for any $-n\leq k\leq n$.
\end{proof}

\subsection{Hochschild cohomology under blowing up}

Let $X$ be a compact complex manifold of dimension $n$.
Recall that the {\it Hochschild cohomology} of $X$ is given by
$$
\mathrm{HH}^{k}(X):=
\mathrm{Ext}_{X\times X}^{k}(\mathcal{O}_{\Delta}, \mathcal{O}_{\Delta}), \;\,0\leq k\leq n.
$$
By the Hochschild-Kostant-Rosenberg Theorem as in \cite[Corollary 4.2]{Cal05},
one has
$$
\mathrm{HH}^{k}(X)
\cong
\bigoplus_{p+q=k} H^{q}(X, \wedge^pT^{\prime}X)
$$
for any $0\leq k\leq n$. Here comes a natural:
\begin{quest}
Is there a blow-up formula for the Hochschild cohomology of compact complex manifolds?
\end{quest}

At present, it seems that it is hard to build an explicit blow-up formula for the Hochschild cohomology; even for smooth projective varieties.
By Hochschild-Kostant-Rosenberg Theorem,
one possible way to handle this question is to understand the cohomology $H^{q}(X, \wedge^pT^{\prime}X)$ under blowing up.
Using our blow-up formula, we will discuss some relations between $H^{q}(X, \wedge^pT^{\prime}X)$ and $H^{q}(\tilde{X}, \wedge^pT^{\prime}{\tilde{X}})$.
First of all, Serre duality gives
$$
H^{q}(X, \wedge^pT^{\prime}X)\cong H^{n-q}(X, \Omega_{X}^{p}\otimes  \omega_{X})^{\ast},
$$
where $\omega_{X}$ is the canonical sheaf of $X$.
Then from the blow-up formula in Theorem \ref{thm1}, it follows
\begin{equation}\label{impHH}
H^{s}(\tilde{X}, \Omega_{\tilde{X}}^{p}\otimes  \pi^{\ast}\omega_{X})
\cong
H^{s}(X, \Omega_{X}^{p}\otimes  \omega_{X})
\oplus \bigoplus_{i=1}^{r-1} H^{s-i}(Z, \Omega_{Z}^{p-i}\otimes  \imath^{\ast}\omega_{X}).
\end{equation}
Set $D:=(r-1)E$ as $r\geq 2$ and the structure sheaf sequence yields  a short exact sequence
 \begin{equation}\label{HHstructsf-seq}
\xymatrix@C=0.5cm{
  0 \ar[r] &  \mathcal{O}_{\tilde{X}}  \ar[r] &  \mathcal{O}_{\tilde{X}}(D)  \ar[r] & j_{\ast}j^{\ast}\mathcal{O}_{\tilde{X}}(D)   \ar[r] & 0,}
\end{equation}
where $j: D\hookrightarrow \tilde{X}$.
Tensoring \eqref{HHstructsf-seq} with $\Omega_{\tilde{X}}^{p}\otimes\pi^{\ast}\omega_{X}$, we get
 \begin{equation}\label{twistHHstructsf-seq}
\xymatrix@C=0.5cm{
  0 \ar[r] & \Omega_{\tilde{X}}^{p}\otimes \pi^{\ast}\omega_{X}  \ar[r] &  \Omega_{\tilde{X}}^{p}\otimes\omega_{\tilde{X}}\ar[r] & \Omega_{\tilde{X}}^{p}\otimes \pi^{\ast}\omega_{X} \otimes j_{\ast}j^{\ast}\mathcal{O}_{\tilde{X}}(D)  \ar[r] & 0},
\end{equation}
since $\omega_{\tilde{X}}=\pi^{\ast}\omega_{X}\otimes \mathcal{O}_{\tilde{X}}((r-1)E)=\pi^{\ast}\omega_{X}\otimes \mathcal{O}_{\tilde{X}}(D)$.
On the long exact sequence \eqref{twistHHstructsf-seq} of sheaf cohomologies,
\eqref{impHH} implies a long exact sequence
$$
\begin{tikzcd}
\cdots\rar & H^{q}(X, \wedge^pT^{\prime}X)^{\ast}\oplus \bigoplus_{i=1}^{r-1} H^{q-i}(Z, \Omega_{Z}^{p-i}(\imath^{\ast}\omega_{X})) \rar & H^{q}(\tilde{X}, \wedge^pT^{\prime}{\tilde{X}})^{\ast}   \ar[out=-10, in=170]{dl} \\
& \quad\quad\quad\quad\quad\quad  H^{q}(D, j^{\ast}(\Omega_{\tilde{X}}^{p}(\pi^{\ast}\omega_{X})\otimes\mathcal{O}_{\tilde{X}}(D)) ) \rar & \cdots.
\end{tikzcd}
$$
In particular, as the codimension $r=2$,
we obtain
$$
\begin{tikzcd}
\cdots\rar & H^{q}(X, \wedge^pT^{\prime}X)^{\ast}\oplus H^{q-1}(Z, \Omega_{Z}^{p-1}(\imath^{\ast}\omega_{X})) \rar & H^{q}(\tilde{X}, \wedge^pT^{\prime}{\tilde{X}})^{\ast}   \ar[out=-10, in=170]{dl} \\
&\quad\quad\quad\quad\quad\quad H^{q}(E, j^{\ast}(\Omega_{\tilde{X}}^{p}(\pi^{\ast}\omega_{X})\otimes\mathcal{O}_{\tilde{X}}(E)) ) \rar & \cdots.
\end{tikzcd}
$$

\appendix
\section{Borel spectral sequence for complex analytic bundles}\label{app-1}

Let $\xi=(E,B,F,\pi)$ be a complex analytic fibre bundle, where $E,B,F$ are connected and $F$ is compact.
Let $W$ be a complex vector bundle on $B$, and $\tilde{W}=\pi^{\ast}W$ the pullback on $E$.
Suppose that $\{f_{\alpha\beta}\}$ is the transition functions of $W$ with respect to a suitable open covering $\{U_{\alpha}\}_{\alpha\in\Lambda}$,
where
$
f_{\alpha\beta}:U_{\alpha}\cap U_{\beta}\rightarrow G.
$
From definition, there exist a representation of $G$ over $H^{p,q}_{\bar{\partial}}(F)$ denoted by $\varphi^{0}:G\rightarrow\textmd{GL}(H^{p,q}_{\bar{\partial}}(F))$.
Consider the collection
$$
\mathfrak{C}^{p,q}:=\bigcup_{b\in B}H^{p,q}_{\bar{\partial}}(F_{b}),
$$
where $F_{b}=\pi^{-1}(b)$.
The collection $\mathfrak{C}^{p,q}$ forms a smooth vector bundle $\mathbf{H}^{p,q}(F)$ over $B$
with the transition functions $\{f^{0}_{\alpha\beta}:=\varphi^{0}\circ f_{\alpha\beta}\}$.
Let $\mathbf{H}_{\bar{\partial}}(F)$ be the direct sum of $\mathbf{H}^{p,q}(F)$, i.e.,
$$
\mathbf{H}_{\bar{\partial}}(F)
=\bigoplus_{p,q}\mathbf{H}^{p,q}(F).
$$
Then we have
\begin{lem}
[{\cite[Appendix Two, \S 1]{Hir66}}]\label{lem3.1}
If $\varphi^{0}$ is constant on the connected components of $G$, in particular if $F$ is K\"{a}hler, then $\mathbf{H}_{\bar{\partial}}(F)$ is a holomorphic vector bundle over $B$.
\end{lem}
In particular, the following theorem attributed to Borel \cite[Appendix Two, 2.1]{Hir66} provides us with an approach to compute the Dolbeault cohomology with values in $\tilde{W}$.
\begin{thm}\label{thm3.2}
Assume that every connected component of the structure group $G$ of $\xi$ acts trivially on $H_{\bar{\partial}}(F)$, and
then there exists a spectral sequence $(E_{k},d_{k})$ ($k\geq0$) with the following properties:
\begin{itemize}
  \item [(i)]$E_{k}$ is 4-graded, by the fibre degree, the base-degree and the type.
              Let $^{p,q}E^{s,t}_{k}$ be the subspace of elements of $E_{k}$ of type $(p,q)$, fibre-degree $s$, base-degree $t$,
              we have $^{p,q}E^{s,t}_{k}=0$ if $p+q\neq s+t$ or if one of $p,q,s,t$ is $<0$.
              The differential $d_{k}$ is
              $$
              d_{k}:{^{p,q}E^{s,t}_{k}}\rightarrow {^{p,q+1}E^{s+k,t-k+1}_{k}}.
              $$
  \item [(ii)] If $p+q=s+t$, we have
              $$
              ^{p,q}E^{s,t}_{2}\cong\sum_{i\geq0}H^{i,s-i}(B,W\otimes\mathbf{H}^{p-i,q-s+1}(F)).
              $$
  \item [(iii)]The spectral sequence converges to $H(E,\tilde{W})$.
              For all $p,q\geq0$ we have
              $$
              \mathrm{Gr}H^{p,q}(E,\tilde{W})
              =\sum_{s+t=p+q}{^{p,q}E^{s,t}_{\infty}}
              $$
              for a suitable filtration of $H^{p,q}(E,\tilde{W})$.
\end{itemize}
\end{thm}

Especially, as a corollary, we have the following result (cf. \cite[Appendix Two, 7.1]{Hir66}).
\begin{cor}\label{cor3.3}
If the vector bundle $\mathbf{H}_{\bar{\partial}}(F)$ is trivial, in particular, if the structure group of $\xi$ is connected,
then we have
$$
^{p,q}E^{s,t}_{2}\cong\sum_{i\geq0}H^{i,s-i}(B,W)\otimes H^{p-i,q-s+i}_{\bar{\partial}}(F).
$$
\end{cor}

\section{Blow-up formula re-examined}\label{app-2}
The purpose of this appendix is to sketch a second proof of Theorem 1.2 by the same philosophy as in \cite{Ste18}.

\subsection{Morphism of cohomology induced by pullback of differential forms}

Suppose that $f:X\rightarrow Y$ is a holomorphic map of complex manifolds.
Then there exists a natural morphism
$\alpha:f^{-1}\Omega^{p}_{Y}\rightarrow\Omega^{p}_{X}$ of $f^{-1}\mathcal{O}_{Y}$-modules,
determined by the pullback of holomorphic $p$-forms.
It is first defined by a morphism of presheaves
$$
\lim_{f(U)\subset V}\Omega_{Y}^{p}(V)
\stackrel{\mu}\longrightarrow
\lim_{f(U)\subset V} \Omega_{X}^{p}(f^{-1}(V))
\stackrel{\nu}\longrightarrow
\Omega_{X}^{p}(U)
$$
for any open sets $U\subset X$ and $V\subset Y$.
Here the first step is the limit of the pullbacks $\mu=\lim_{f(U)\subset V}f^{\ast}|_{f^{-1}(V)}$
and $\nu$ is the restriction.
Note that
$f_{\ast}\Omega_{X}^{p}(V)= \Omega_{X}^{p}(f^{-1}(V))$.
By the uniqueness of sheafification,
we obtain a composition morphism of $f^{-1}\mathcal{O}_{Y}$-modules
\begin{equation}\label{topolpullbackmorph}
\alpha:
f^{-1}\Omega_{Y}^{p}
\longrightarrow
f^{-1}f_{\ast}\Omega_{X}^{p}
\longrightarrow
\Omega_{X}^{p}.
\end{equation}

Because $f^{-1}$ is the left adjoint of $Rf_{\ast}$,
there is a functorial morphism
\begin{equation}\label{funct-morph}
\mathcal{E}\longrightarrow Rf_{\ast}f^{-1}\mathcal{E}
\end{equation}
for any $\mathcal{E}\in \mathrm{D}^{b}(\mathcal{O}_{Y})$ the bounded derived category of $\mathcal{O}_{Y}$-modules (cf. \cite[Proposition 2.6.4]{KS94}).
Combining \eqref{topolpullbackmorph}
with the functorial morphism \eqref{funct-morph} for $\Omega_{Y}^{p}$ and $f_{\ast}\Omega_{X}^{p}$ leads to
the following morphisms:
\begin{equation}\label{real-pullback}
\Omega_{Y}^{p}
\longrightarrow
Rf_{\ast}f^{-1}\Omega_{Y}^{p}
\longrightarrow
Rf_{\ast}\Omega_{X}^{p}
\end{equation}
and
\begin{equation}\label{real-edge-morph}
f_{\ast}\Omega_{X}^{p}
\longrightarrow
Rf_{\ast}f^{-1}f_{\ast}\Omega_{X}^{p}
\longrightarrow
Rf_{\ast}\Omega_{X}^{p}
\end{equation}
in $\mathrm{D}^{b}(\mathcal{O}_{Y})$.

Note that there exists a natural morphism
$f^{\ast}: \Omega_{Y}^{p}\rightarrow f_{\ast} \Omega_{X}^{p}$
induced by pullback of differential forms.
By \eqref{topolpullbackmorph}-\eqref{real-edge-morph},
we get a commutative diagram
\begin{equation*}\label{compare}
\xymatrix@C=0.5cm{
f^{\ast}: &H^{l}(Y, \Omega_{Y}^{p}) \ar[d]_{f^{\ast}} \ar[r]  &
H^{l}(X, f^{-1}\Omega_{Y}^{p}) \ar[d]_{f^{\ast}}  \ar[r]^{\alpha}& H^{l}(X, \Omega_{X}^{p}) \ar[d]_{=} \\
f^{\dag}: & H^{l}(Y, f_{\ast}\Omega_{X}^{p}) \ar[r] & H^{l}(X, f^{-1}f_{\ast}\Omega_{X}^{p}) \ar[r] & H^{l}(X, \Omega_{X}^{p}). &}
\end{equation*}

By the construction, it is important to notice that $f^{\ast}: H^{l}(Y, \Omega_{Y}^{p})
\rightarrow
H^{l}(X, \Omega_{X}^{p})$
is the morphism induced by pullback of differential forms.
Moreover, the morphism $f^{\dag}$ is the edge morphism of the Leray spectral sequence for $\Omega_{X}^{p}$ under the morphism $f$ (cf. \cite[(13.8) Theorem of Chapter IV]{Dem12}). The bundle-valued case can be dealt with similarly.

\subsection{Sketch of proof}
Note that the morphisms
$
\pi^{\ast}:H^{p,q}(X,W)\rightarrow H^{p,q}(\tilde{X},\tilde{W})
$
and
$
\varpi^{\ast}:H^{p,q}(Z,\imath^{\ast}W)
\rightarrow
H^{p,q}(E,\jmath^{\ast}\tilde{W})
$
are injective because of \cite[Theorem 3.1]{Wel74} and Lemma \ref{dolb-projective-formula}.
The blow-up diagram \eqref{blowup-diag3} induces a commutative diagram of short exact sequences
\begin{equation}\label{comm-coker-C}
\xymatrix@C=0.5cm{
0 \ar[r]^{} & H^{p,q}(X,W)
\ar[d]_{\imath^{\ast}}
\ar[r]^{\pi^{\ast}} &
H^{p,q}(\tilde{X},\tilde{W})
\ar[d]_{\jmath^{\ast}}
\ar[r]^{} &
\mathrm{coker}\,(\pi^{\ast})
\ar[d]_{\bar{\jmath}^{\ast}} \ar[r]^{} & 0 \\
0 \ar[r] &
H^{p,q}(Z,\imath^{\ast}W)
\ar[r]^{\varpi^{\ast}} &
H^{p,q}(E,\jmath^{\ast}\tilde{W})
\ar[r]^{} &
\mathrm{coker}\,(\varpi^{\ast})\ar[r]^{} & 0.}
\end{equation}
Like \eqref{comm-coker},
to prove the blow-up formula, it remains to verify that $\bar{\jmath}^{\ast}$ is isomorphic.

Consider the sheaves $\Omega^{p}_{\tilde{X}}(\tilde{W})$
and $\Omega^{p}_{E}(\jmath^{\ast}\tilde{W})$ on $\tilde{X}$ and $E$, respectively.
We have the associated Leray spectral sequences
$\{E_{k}(\tilde{W}),d_{k}\}$
and
$\{E_{k}(\jmath^{\ast}\tilde{W}),d^{\prime}_{k}\}$
with the second terms
$$
E^{l,s}_{2}(\tilde{W})=
H^{l}(X,R^{s}\pi_{\ast}\Omega^{p}_{\tilde{X}}(\tilde{W}))
\Longrightarrow
\mathrm{Gr}\,
H^{p,l+s}(\tilde{X},\tilde{W})
$$
and
$$
E^{l,s}_{2}(\jmath^{\ast}\tilde{W})
=H^{l}(Z,R^{s}\varpi_{\ast}\Omega^{p}_{E}(\jmath^{\ast}\tilde{W}))
\Longrightarrow
\mathrm{Gr}\,H^{p,l+s}(E,\jmath^{\ast}\tilde{W}).
$$
Here the filtration of $H^{p,l+s}(-)$ is defined by setting
$$
H^{p,l+s}(-)=F^{0}H^{p,l+s}(-)\supset
F^{1}H^{p,l+s}(-)\supset\cdots\supset
F^{l+s}H^{p,l+s}(-)\supset F^{l+s+1}H^{p,l+s}(-)=0
$$
and the associated graded space is defined by
$$
\mathrm{Gr}^{k}(-)=\frac{F^{k}(-)}{F^{k+1}(-)}.
$$
From Lemma \ref{key-lem-1} we have
$$
E^{l,0}_{2}(\tilde{W})
=H^{l}(X,\pi_{\ast}\Omega^{p}_{\tilde{X}}(\tilde{W}))
\cong
H^{l}(X,\Omega^{p}_{X}(W))
$$
and
$$
E^{l,0}_{2}(\jmath^{\ast}\tilde{W})
=H^{l}(Z,\varpi_{\ast}\Omega^{p}_{E}(\jmath^{\ast}\tilde{W}))
\cong H^{l}(Z,\Omega^{p}_{Z}(\imath^{\ast}W)).
$$
On one hand, the Dolbeault theorem implies
$$
H^{l}(\tilde{X},\Omega^{p}_{\tilde{X}}(\tilde{W}))\cong H^{p,l}(\tilde{X},\tilde{W})
\,\,\,\mathrm{and}\,\,\,
H^{l}(X,\Omega^{p}_{X}(W))
\cong
H^{p,l}(X,W).
$$
On the other hand, the pullback of the differential forms determines a morphism
$$
\pi^{\ast}:
H^{p,l}(X,W)
\rightarrow
H^{p,l}(\tilde{X},\tilde{W}).
$$
So one can have the following commutative diagram of abelian groups
$$
\xymatrix@R=0.4cm@C=0.3cm{
&E^{l,0}_{\infty}(\tilde{W})\ar@{^{(}->}[dr]& \\
H^{l}(X,\Omega^{p}_{X}(W))
=E^{l,0}_{2}(\tilde{W})
\ar@{->>}[ru]\ar[d]_{\cong}^{\mathrm{Dolbeault}}
\ar[rr]^{\pi^{\dag}}& &H^{l}(\tilde{X},\Omega^{p}_{\tilde{X}}(\tilde{W}))
\ar[d]_{\cong}^{\mathrm{Dolbeault}}\\
H^{p,l}(X,W)\ar[rr]^{\pi^*}&&
H^{p,l}(\tilde X,\tilde W)}
$$
where $\pi^{\dag}$ is the edge morphism of $E^{l,0}_{2}(\tilde{W})$.
This implies that the differentials for $E(\tilde{W})$ with target of degree $(l,0)$ are zero.
Likewise, we can show that the differentials for
$E(\jmath^{\ast}\tilde W)$ with target of degree $(l,0)$ are also zero.
In particular, we have an exact sequence
\begin{equation}\label{exact-seq}
\xymatrix@C=0,5cm{
  0 \ar[r] & E^{q,0}_{\infty}(\tilde{W})
  \ar[r]^{\subseteq\quad\quad} & H^{q}(\tilde{X},\Omega^{p}_{\tilde{X}}(\tilde{W}))
  \ar[r]^{} &
  H^{q}(\tilde{X},\Omega^{p}_{\tilde{X}}(\tilde{W}))
  /E^{q,0}_{\infty}(\tilde{W})
   \ar[r] & 0. }
\end{equation}
By definition, there hold the equalities
$$
E^{q,0}_{\infty}(\tilde{W})=
\frac{F^{q}
H^{q}(\tilde{X},\Omega^{p}_{\tilde{X}}(\tilde{W}))}
{F^{q+1}
H^{q}(\tilde{X},\Omega^{p}_{\tilde{X}}(\tilde{W}))}
=F^{q}
H^{q}(\tilde{X},\Omega^{p}_{\tilde{X}}(\tilde{W}))
\cong
H^{p,q}({X},{W})
$$
and
$$
F^{0}
H^{q}(\tilde{X},\Omega^{p}_{\tilde{X}}(\tilde{W}))
=
H^{q}(\tilde{X},\Omega^{p}_{\tilde{X}}(\tilde{W}))
\cong
H^{p,q}(\tilde{X},\tilde{W}).
$$
So the exact sequence \eqref{exact-seq} becomes
$$
\xymatrix@C=0.5cm{
  0 \ar[r]^{} & F^{q}H^{p,q}(\tilde{X},\tilde{W})
  \ar[r]^{\subseteq} & F^{0}H^{p,q}(\tilde{X},\tilde{W}) \ar[r]^{} &
  F^{0}H^{p,q}(\tilde{X},\tilde{W})/
  F^{q}H^{p,q}(\tilde{X},\tilde{W}) \ar[r]^{} & 0,}
$$
which is equal to
$$
\xymatrix@C=0.5cm{
0 \ar[r]^{} & H^{p,q}(X,W)
\ar[r]^{\pi^{\ast}} &
H^{p,q}(\tilde{X},\tilde{W})
\ar[r]^{} &
\mathrm{coker}\,(\pi^{\ast})
\ar[r]^{} & 0.}
$$
Similarly, we can show that
$$
\xymatrix@C=0.5cm{
0 \ar[r] &
H^{p,q}(Z,\imath^{\ast}W)
\ar[r]^{\varpi^{\ast}} &
H^{p,q}(E,\jmath^{\ast}\tilde{W})
\ar[r]^{} &
\mathrm{coker}\,(\varpi^{\ast})\ar[r]^{} & 0}
$$
equals to
$$
\small{
\xymatrix@C=0.5cm{
  0 \ar[r]^{} & F^{q}H^{p,q}(E,\jmath^{\ast}\tilde{W})
  \ar[r]^{\subseteq} & F^{0}H^{p,q}(E,\jmath^{\ast}\tilde{W}) \ar[r]^{} &
  F^{0}H^{p,q}(E,\jmath^{\ast}\tilde{W})/
  F^{q}H^{p,q}(E,\jmath^{\ast}\tilde{W}) \ar[r]^{} & 0. }}
$$
Consequently, \eqref{comm-coker-C} is equivalent to
\begin{equation*}\label{spec-coker}
\xymatrix@C=0.4cm{
  0 \ar[r]^{} & F^{q}H^{p,q}(\tilde{X},\tilde{W})
  \ar[d]\ar[r]^{\subseteq} & F^{0}H^{p,q}(\tilde{X},\tilde{W}) \ar[d]\ar[r]^{} &\ar[d]_{\jmath^{\ast}}
  F^{0}H^{p,q}(\tilde{X},\tilde{W})/
  F^{q}H^{p,q}(\tilde{X},\tilde{W}) \ar[r]^{} & 0 \\
   0 \ar[r]^{} & F^{q}H^{p,q}(E,\jmath^{\ast}\tilde{W})
  \ar[r]^{\subseteq} & F^{0}H^{p,q}(E,\jmath^{\ast}\tilde{W}) \ar[r]^{} &
  F^{0}H^{p,q}(E,\jmath^{\ast}\tilde{W})/
  F^{q}H^{p,q}(E,\jmath^{\ast}\tilde{W}) \ar[r]^{} & 0.}
\end{equation*}
Moreover, due to the third isomorphism in Lemma \ref{key-lem-1} one gets
$$
\jmath^{\ast}:E^{l,s}_{2}(\tilde{W})\cong E^{l,s}_{2}(\jmath^{\ast}\tilde{W})
\,\,\mathrm{for\,\,any}\,\,s\geq1,
$$
and thus, for any $l$ with $q=l+s$, the following isomorphism holds:
$$
\jmath^{\ast}:\mathrm{Gr}^{l}
H^{p,q}(\tilde{X},\tilde{W})
\cong
\mathrm{Gr}^{l}
H^{p,q}(E,\jmath^{\ast}\tilde{W}).
$$
As a result, one has
$$
\jmath^{\ast}:
\frac{F^{0}
H^{p,q}(\tilde{X},\tilde{W})}
{F^{q}
H^{p,q}(\tilde{X},\tilde{W})}
\cong
\frac{F^{0}H^{p,q}(E,\jmath^{\ast}\tilde{W})}
{F^{q}H^{p,q}(E,\jmath^{\ast}\tilde{W})}
$$
and therefore $\bar{\jmath}^{\ast}$ in \eqref{comm-coker-C} is an isomorphism.
\begin{rem} It is easy to check that the isomorphisms \eqref{crd-iso}
\begin{equation*}
\pi^{\ast}: H^{q}(X, \K_{X,Z}^{p}(W)) \stackrel{\simeq}\longrightarrow H^{q}(\tilde{X}, \K_{\tilde{X},E}^{p}(\tilde{W}))
\end{equation*}
for all suitable $p,q$
are equivalent to the ones \eqref{comm-coker-C}
$$
\frac{H^{p,q}(\tilde{X},\tilde W)}{\pi^{*}H^{p,q}(X,W)}
\cong \frac{H^{p,q}(E,\jmath^*\tilde{W})}{\varpi^{*}H^{p,q}(Z,\imath^*{W})}
$$
by use of the commutative diagram
\begin{equation*}
\small{
\xymatrix@=0.5cm{
     \cdots\ar[r]&\ker\,(\bar{\pi}^{*}_{q}) \ar[d]\ar[r] & 0  \ar[d]_{}\ar[r] & 0\ar[d]_{} \ar[r] &  \ker\,(\bar{\pi}^{*}_{q+1}) \ar[d]\ar[r] & \cdots \\
   \cdots \ar[r]^{} & H^{q}(X,\mathscr{K}^{p}_{X,Z}(W)) \ar[d]_{\bar{\pi}^{\ast}_{q}} \ar[r]^{}&
   H^{p,q}(X,W) \ar[d]_{\pi^{\ast}_{q}} \ar[r]^{\imath^{\ast}}
   & H^{p,q}(Z,\imath^{*}W) \ar[d]_{\varpi^{\ast}_{q}} \ar[r]^{\delta\quad} &H^{q+1}(X,\mathscr{K}^{p}_{X,Z}(W)) \ar[d]_{\bar{\pi}^{\ast}_{q+1}} \ar[r]^{}
   & \cdots \\
   \cdots \ar[r] & H^{q}(\tilde{X}, \mathscr{K}^{p}_{\tilde{X},E}(\tilde{W}))
   \ar[d]\ar[r]^{} &
   H^{p,q}(\tilde{X},\tilde{W})\ar[d] \ar[r]^{\jmath^{\ast}} &
   H^{p,q}(E,\jmath^{*}\tilde{W})\ar[d] \ar[r]^{\tilde{\delta}\quad}&
  H^{q+1}(\tilde{X}, \mathscr{K}^{p}_{\tilde{X},E}(\tilde{W}))  \ar[d]\ar[r]^{} &\cdots\\
  \cdots \ar[r] & \mathrm{coker}\,(\bar{\pi}^{*}_{q})
 \ar[d]\ar[r]^{}  &
  \mathrm{coker}\,(\pi^{*}_{q})\ar[d]\ar[r]
  &
  \mathrm{coker}\,(\varpi^{*}_{q})
  \ar[d]\ar[r]^{\hat{\delta}\quad}
  &\mathrm{coker}\,(\bar{\pi}^{*}_{q+1})
  \ar[r]^{} \ar[d]&\cdots\\
  &0  &0  &0  &0  }
  }
\end{equation*}
See also the simple Dolbeault case in \cite[Proof of Proposition 3.4 and Remark 3.6]{RYY}$_{v4}$.
\end{rem}


\end{document}